\documentclass[a4paper,10pt,         
pointlessnumbers,      
twoside,
DIV=calc,               
BCOR15mm               
]{article}
\usepackage{geometry}
\usepackage[T1]{fontenc}        

\usepackage[utf8x]{inputenc}    
\usepackage{ucs}                
\usepackage{amsmath}            
\usepackage{amsfonts}
\usepackage{amssymb}
\usepackage{amsthm} 
\usepackage{thmtools}
\usepackage[authoryear]{natbib}
\usepackage{ifthen}
\usepackage{color}
\usepackage{xcolor}	
\usepackage{colortbl}
\usepackage{graphicx} 
\usepackage{textcomp}            
\usepackage{palatino}            
\usepackage{hhline}

\usepackage{eulervm}            
\usepackage{varioref}            
\usepackage[colorlinks=false,plainpages=false]{hyperref}
\usepackage{url}                
\usepackage{cleveref}             
\usepackage{fancyhdr}
\usepackage{float}
\usepackage{multirow}
\usepackage{rotating}
\usepackage{subfig}
\usepackage{titlesec}
\usepackage{setspace}
\usepackage{tabularx}
\usepackage{booktabs}
\definecolor{tugreen}{HTML}{84b819}
\definecolor{pgrey}{rgb}{242, 242, 242}
\definecolor{ashgrey}{rgb}{0.7, 0.75, 0.71}
 	\definecolor{gainsboro}{rgb}{0.86, 0.86, 0.86}
 	\definecolor{lightyellow}{rgb}{1.0, 1.0, 0.88}
 	\definecolor{pastelgray}{rgb}{0.81, 0.81, 0.77}
 	\definecolor{timberwolf}{rgb}{0.86, 0.84, 0.82}
 	\definecolor{beige}{rgb}{0.96, 0.96, 0.86}
\titleformat{\chapter}{\fontfamily{qzc}\selectfont \Huge \color{tugreen}}{\thechapter}{1ex}{}
\titlespacing{\chapter}{0pt}{-2ex}{*1.5}
\usepackage{xr} 
\usepackage{pdfpages}
\usepackage{xcite}

\newcommand{\R}{{\mathbb{R}}}         
\newcommand{\E}{{\mathbb{E}}}
 \newcommand{\N}{{\mathbb{N}}}         

\newcommand{\bay}{\begin{array}}
\newcommand{\eay}{\end{array}}

\newcommand{\bqa}{\begin{eqnarray*}}
\newcommand{\eqa}{\end{eqnarray*}}

\newcommand{\bee}{\begin{eqnarray*}}
\newcommand{\eee}{\end{eqnarray*}}

\newcommand{\bea}{\begin{eqnarray*}}
\newcommand{\eea}{\end{eqnarray*}}

\newcommand{\bqan}{\begin{eqnarray}}
\newcommand{\eqan}{\end{eqnarray}}

\newcommand{\be}{\begin{eqnarray}}
\newcommand{\ee}{\end{eqnarray}}

\newcommand{\bit}{\begin{itemize}}
\newcommand{\eit}{\end{itemize}}

\newcommand{\ben}{\begin{enumerate}}
\newcommand{\een}{\end{enumerate}}

\newcommand{\beq}{\begin{equation}}
\newcommand{\eeq}{\end{equation}}

\newcommand{\bdes}{\begin{description}}
\newcommand{\edes}{\end{description}}

\newcommand{\btb}{\begin{tabular}}
\newcommand{\etb}{\end{tabular}}

\newcommand{\bcen}{\begin{center}}
\newcommand{\ecen}{\end{center}}

\newcommand{\bmp}{\begin{minipage}}
\newcommand{\emp}{\end{minipage}}

\newcommand{\Cov}{\operatorname{{\it Cov}}}

\newcommand{\Var}{\operatorname{{\it Var}}}

\newcommand{\tr}{\operatorname{tr}}

\newcommand{\vH}{\boldsymbol{H}}
\newcommand{\vI}{\boldsymbol{I}}
\newcommand{\vJ}{\boldsymbol{J}}

\newcommand{\vP}{\boldsymbol{P}}

\newcommand{\vT}{\boldsymbol{T}}

\newcommand{\vV}{\boldsymbol{V}}

\newcommand{\vX}{\boldsymbol{X}}

\newcommand{\vZ}{\boldsymbol{Z}}

\newcommand{\vmu}{\boldsymbol{\mu}}

\newcommand{\vpi}{\boldsymbol{\pi}}

\newcommand{\vsigma}{\boldsymbol{\sigma}}
\newcommand{\vSigma}{\boldsymbol{\Sigma}}

\newcommand{\vell}{\boldsymbol{\ell}}

\newcommand{\veins}{{\bf 1}}
\newcommand{\vnull}{{\bf 0}}

\newcommand{\vUpsilon}{\boldsymbol{\Upsilon}}

\newcommand{\ind}{1\hspace{-0.7ex}1}

\newcommand{\F}{\mathcal {F}}

\newcommand{\Cdot}{\cdot}

\newtheoremstyle{Test1}
  {2 \baselineskip}
  {1.5 \baselineskip}
  {\itshape}
  {-0.0ex}
  {\fontfamily{ppl}\fontseries{l}\fontshape{n}}
  {:}
  {\newline}
   {}

\theoremstyle{Test1}
\newtheorem{Sa}{Theorem}[section]
\newtheorem{theorem}{Theorem}[section]
\newtheorem{re}[Sa]{Remark}
 
\newtheorem{Le}[Sa]{Lemma}

\newcommand{\Lan}{\mathcal{O}}
\newcommand{\lan}{ \scriptstyle \mathcal{O}\textstyle}

\newcolumntype{x}[1]{!{\centering\arraybackslash\vrule width #1}}

\makeatletter
\renewenvironment{proof}[1][\proofname]{\par
  \pushQED{\qed}%
  \fontfamily{ppl}\fontseries{m}\fontshape{it} \topsep6\p@\@plus6\p@\relax
  \trivlist
  \item[\hskip\labelsep
        \bfseries
    #1\@addpunct{:}]\ignorespaces
}{%
  \popQED\endtrivlist\@endpefalse
}
\makeatother

\begin{document}
\title{\Large \bf Inference for high-dimensional split-plot designs with different dimensions between groups
}
\author{Paavo Sattler$^{1}$ and Markus Pauly$^{1}$ 
\\[0.4ex] }

\vspace{-8ex}
  \date{}
\maketitle \vspace*{0.25ex}
\begin{center}\noindent${}^{1}$ {TU Dortmund University, Faculty of Statistics, Germany\\\mbox{ }\hspace{1 ex}email: paavo.sattler@tu-dortmund.de  }\\
 
 \end{center}
\begin{abstract}
\noindent
In repeated Measure Designs with multiple groups, the primary purpose is to compare different groups in various aspects. For several reasons, the number of measurements and therefore the dimension of the observation vectors can depend on the group, making the usage of existing approaches impossible. We developed an approach which can be used not only for a possibly increasing number of groups $a$, but also for group-depending dimension $d_i$, which is allowed to go to infinity. This is a unique high-dimensional asymptotic framework impressing through its variety and do without usual conditions on the relation between sample size and dimension. It especially includes settings with fixed dimensions in some groups and increasing dimensions in other ones, which can be seen as semi-high-dimensional.\\
To find a appropriate statistic test new and innovative estimators are developed, which can be used under these diverse settings on $a,d_i$ and $n_i$  without any adjustments. We investigated the asymptotic distribution of a quadratic-form-based test statistic and developed an asymptotic correct test. Finally, an extensive simulation study is conducted to investigate the role of the single group's dimension.
\end{abstract}

\noindent{\textbf{Keywords:}} High-dimensional, Quadratic Forms, split-plot-designs different dimensions.

\vfill
\vfill

\section{MOTIVATION AND INTRODUCTION}\label{int}
In repeated measure designs, measurements were made repeatedly on the same object or subject, leading to dependence between these values. If the observations, represented through $d$ dimensional vectors, are compared for different groups, we talk about a so-called split-plot design. It presupposes that the number of repetitions in each group is the same, which often does not correspond with reality. For example, for questionnaires with group-specific questions (like pregnancy for women),  the number of questions or answers depends understandably on the group. Another example are time series, where for organizational reasons, in some groups, the number of repetitions is smaller than in others, or one group has fewer treatments. 
Moreover, the influence of the number of repetitions on the result is frequently investigated. The impact of a questionnaire's length on the results is a popular topic, see, e.g., \cite{roszkowski1990} or \cite{hallal2014}. But it also enables, for example, to investigate whether the number of animals in a litter or a cage influences the respective development status like it was done in \cite{hughes1973} or \cite{chvedoff1980}.
In such cases, the group membership depends primarily on the length of the respective observation vector. Settings with different dimensions in different groups have only scarcely been investigated up to now but were, for example, considered in \cite{friedrich2017permuting}. 
High dimensional split plot designs are the subject of many works like \cite{brunner2012},\cite{happ2016} or \cite{harrar2016}. Moreover, \cite{sattler2018} and \cite{sattler2021} also considered frameworks,where in addition to the dimension and the sample size, the number of groups is allowed to increase.
But none of these works allow a group depending dimension $d_i$, for the i-th group. Therefore we aimed to expand the work of \cite{sattler2018} for this setting. Thereby the previous family of hypothesis matrices has to be reconsidered. Moreover, it is clear that some hypotheses like equality of means make no sense in such situations. \\
  The presented paper is organized as follows. In \Cref{mod}, we introduce the statistical model, the considered hypotheses and the used asymptotic frameworks. Based on this in \Cref{Test}, a test statistic, necessary estimators and an additional small sample approximation are developed.
To investigate the behaviour of the developed test, in \Cref{simulations} simulations regarding the type-I-error rate were done for different settings. This paper closes with a short conclusion. 
For better readability and brevity, all proofs and some further estimators are shifted to the appendix.

\section{\textsc{Statistical Model and Hypotheses}} \label{mod}
We want to adapt the model and the results of \cite{sattler2018} for the more general setting with different dimensions in the groups. Such a split-plot design was also considered in \cite{friedrich2017permuting}. We assume for positiv definite covariance matrices $\vSigma_i\in \R^{d_i\times d_i}$,
\[{\vX}_{i,j}= ({X}_{i,j,1},\dots,{X}_{i,j,d_i})^\top \stackrel{ind}{\sim}\mathcal{N}_{d_i}\left(\vmu_i,\vSigma_i\right)\hspace{0,2cm}j=1,\dots,n_i,\hspace{0,2cm} i=1,\dots,a ,\]
whereby each vector represents the measurement of one independent subject or object. As usual with $n_i\in \N$ observation vectors in $a\in \N$ groups $i=1,\dots, a$, we have a total number of $N=\sum_{i=1}^a n_i$ random vectors. Moreover the dimension of the pooled vector is denoted by $D=\sum_{i=1}^a d_i$. This framework also allows a factorial structure by splitting up the indices,  regarding time, group or both, as done in \cite{kon:2015} for example.
 In this work, the usual condition for designs with several groups $n_i/{N}\to \kappa_i\in(0,1]$ for  $i=1,...,a$ is not necessary.\\

Since a Kronecker product of wholeplot and subplot matrix can not be used in a model allowing groups with different dimensions, we consider another class of null hypotheses.

Thereto, a block matrix \[\vT=\begin{pmatrix}
\vT_{11}&...&\vT_{1a}\\
\vdots&\ddots&\vdots\\
\vT_{a1}&\cdots&\vT_{aa}
\end{pmatrix}\in \R^{D\times D}\] which is idempotent and symmetric with components $\vT_{ij}\in \R^{d_i\times d_j}$ is used to formulate our null hypothesis through $\mathcal{H}_0:\vT\vmu=\vnull_D$. Unfortunately the idempotence and symmetry of the components do not follow from the fact that $\vT$ is a projection matrix and vice versa. Therefore we can not assume the parts of the hypothesis matrix to be quadratic, neither idempotent or symmetric. However, through the symmetry of $\vT$ it holds $\vT_{ij}=\vT_{ji}^\top$ for $i,j \in \N_a:=\{1,...,a\}$. One important hypothesis, which is part of this model, compares the averaged value for the repeated measurement between two groups.

Finally, the so far used asymptotic frameworks have to be reconsidered for the case of different dimensions between the groups. An essential element of the approach from \cite{sattler2018} is that the dimension of the pooled mean vector goes to infinity. For equal dimensions, this was $ad$ and therefore, either the dimension or number of groups needed to go to infinity. But in this new setting, $ad$ is replaced by $D$. Thus, the number of groups or at least one dimension has to go to infinity, while the other ones could be fixed. This allows very unbalanced dimensions and settings with fixed dimensions in some groups and increasing dimensions in others, which can be seen as semi-high-dimensional. As a result, comparing data sets from trials with fixed dimensions with high-dimensional data sets is possible. To our knowledge, such a comparison has not yet been part of other papers and therefore offers many possibilities.\\

 So the new, more general frameworks are

\begin{eqnarray}
a\in\N\; \text{ fixed and}&\phantom{1}\hspace*{0.2cm} \min(\max(d_1,...,d_a),n_1,...,n_a)\to \infty, \label{asframe1}\\
\forall i \in \N_a,\  d_i\in\N\; \text{ fixed }   \text{and}&\phantom{1}\hspace*{2.1cm}  \min(a,n_1,...,n_a))\to \infty, \label{asframe2}\\
\text{or}& \min(a,\max(d_1,...,d_a),n_1,...,n_a)\to \infty, \label{asframe3}
\end{eqnarray}

especially including the semi-high-dimensional settings as well as all frameworks from \cite{sattler2018} for the special case $d_i\equiv d$.

\section{Test statistics and their asymptotics}\label{Test}

To investigate the validity of the  {null hypothesis} $H_0(\vT):\vT\vmu=\vnull$ in these asymptotic frameworks, we use
$Q_N=N\cdot\overline {\vX}^\top \vT\overline {\vX}.$
 Here  the vector of pooled group means is denoted by ${\overline{\vX}=(\overline\vX_1^\top,\dots\overline\vX_a^\top)^\top}$ with $\overline {\vX}_{i} = n_i^{-1} \sum_{j=1}^{n_i} \vX_{i,j}, i=1,\dots,a$. 
Unfortunately, the random variable {$Q_N$} tends to converge to infinity, for $d\to \infty$ or $a\to \infty$ and many covariance matrices $\vSigma_i$. We use  
 the standardized quadratic form {given by}

\[\widetilde W_N=\frac{Q_N-\E_{\mathcal H_0}(Q_N)}{\sqrt{\Var_{\mathcal H_0}(Q_N)}},\]
to avoid this behaviour, which also enables us to evaluate all limit distributions in detail.\\

With  $\vV_N=\bigoplus_{i=1}^a \frac{N}{n_i}\vSigma_i$ we know 
$\sqrt{N}\ \vT\overline \vX\stackrel{\mathcal H_0}{\sim}\mathcal{N}_D\left(\vnull_D,\vT\vV_N\vT\right)$, and therefore, the expectation and variance of the quadratic form is known and can be calculated as

 \begin{eqnarray}\E_{\mathcal H_0}(Q_N)&=&\tr(\vT\vV_N \vT)=\tr(\vT\vV_N)=\sum\limits_{i=1}^a\frac{N}{n_i}\tr(\vT_{ii}\vSigma_{ii})\label{EQF}\\
\Var_{\mathcal H_0}(Q_N)&=&2\tr((\vT\vV_N)^2)=2 \sum\limits_{i=1}^a\sum\limits_{r=1}^a \frac{N^2}{n_i n_r} \tr\left(\vT_{ir}\vSigma_r \vT_{ri}\vSigma_i\right).\label{VQF}\end{eqnarray}\\

Through the normal distributed mean vectors, with these traces we can reformulate the standardized quadratic form $\widetilde W_N$ through the representation theorem for quadratic forms (\cite{mathaiProvost}) as
\[\widetilde W_N=\frac{Q_N-\tr(\vT\vV_N) }{\sqrt{2\tr((\vT\vV_N)^2)}}\stackrel{\mathcal{D}}{=} \ \sum\limits_{s=1}^{D} 
\frac{\lambda_s}{\sqrt{\sum_{\ell=1}^{D} \lambda_\ell^2}}\left(\frac{C_s-1}{\sqrt{2}}\right).\]

Here $\lambda_s$ are the eigenvalues of $\vT\vV_N\vT$ in decreasing order, and $(C_s)_s$ is a sequence of independent 
$\chi_1^2$-distributed random variables. This representation allows calculating the asymptotic distribution and, in this way, defining asymptotic correct tests.

\begin{theorem}\label{Asymptotik}
{Let $\beta_s={\lambda_s}\Big/{\sqrt{\sum_{\ell=1}^{D}\lambda_\ell^2}}$ for $s=1,\dots,D$. }
Then $\widetilde W_N$ has, under $\mathcal H_0(\vT)$, and one of the 
frameworks  \eqref{asframe1}-\eqref{asframe3}  asymptotically
\begin{itemize} 
\item[a)]a distribution of the form $\sum_{s=1}^rb_s \left(C_s-1\right)/\sqrt 2+\sqrt{1-\sum_{s=1}^r b_s^2}\cdot Z$,  if and only if
\[\text{for all } s\in \N \hspace{0.5cm}\beta_s\to b_s \hspace{0.5cm} \text{as }\hspace{0.2cm} N \to \infty,\]
for a decreasing sequence $(b_s)_s$ in $[0,1]$ with  $r:=\#\{b_s\neq 0\}$, while $C_s\stackrel{i.i.d.}{\sim} \chi_1^2$, $Z\sim \mathcal{N}(0,1)$.
\item[b)]a distribution of the form $\sum_{s=1}^\infty b_s \left(C_s-1\right)/\sqrt 2$,  if 
\[\text{for all } s\in \N \hspace{0.5cm}\beta_s\to b_s \hspace{0.5cm} \text{as }\hspace{0.2cm} N \to \infty,\]
for a decreasing sequence $(b_s)_s$ in $(0,1)$ with $\sum_{s=1}^\infty b_s^2=1$ and $C_s\stackrel{i.i.d.}{\sim} \chi_1^2$.

\end{itemize}
\end{theorem}

Since usual dimensional stable estimators for traces require idempotence and symmetry, it is not possible to estimate each of the traces from \eqref{EQF} and \eqref{VQF} separately. Therefore, with $\vell_1=(\ell_{1,1},...,\ell_{1,a})$ and $\vell_2=(\ell_{2,1},...,\ell_{2,a})$ we use a vector $\vZ$ given by

 \[\vZ_{(\vell_1,\vell_2)}:=\left(\sqrt{\frac{N}{n_1}}\left(\vX_{1,\ell_{1,1}}-\vX_{1,\ell_{2,1}}\right)^\top\textbf{,}\dots\textbf{,}\sqrt{\frac{N}{n_a}}\left(\vX_{a,\ell_{1,a}}-\vX_{a,\ell_{2,a}}^\top\right)^\top\right), \]
as it was done in \cite{sattler2018}.
For $\ell_{1,1}\neq \ell_{2,1} \in \N_{n_1},....,\ell_{1,a}\neq \ell_{2,a} \in \N_{n_a}$, it then holds $\vT\vZ\sim\mathcal{N}_D\left(\vnull_D,2 \vT\vV_N\vT\right)$. Therefore unbiased trace estimators for 
$\tr(\vT\vV_N\vT)$ resp. $\tr\left(\left(\vT\vV_N\vT\right)^2\right)$ are given through

\[A_{1}=\sum\limits_{\begin{footnotesize}\substack{\ell_{1,1},\ell_{2,1}=1\\\ell_{1,1}< \ell_{2,1}}\end{footnotesize}}^{n_1}
\dots\sum\limits_{\begin{footnotesize}\substack{\ell_{1,a}, \ell_{2,a}=1\\\ell_{1,a}<\ell_{2,a}}\end{footnotesize}}^{n_a}
\frac{{\vZ_{(\vell_1,\vell_2)}}^\top \vT{\vZ_{(\vell_1,\vell_2)}}}{2\cdot \prod_{i=1}^a\binom{n_i}{2} }\]
and
\[A_{2}=\sum\limits_{\begin{footnotesize}\substack{\ell_{1,1}\neq\dots\neq \ell_{4,1}=1\\\ell_{1,1}<\ell_{1,1},\dots,\ell_{3,1}<\ell_{4,1}}\end{footnotesize}}^{n_1}
\dots\sum\limits_{\begin{footnotesize}\substack{\ell_{1,a}\neq\dots\neq \ell_{4,a}=1\\\ell_{1,a}<\ell_{2,a},\dots,\ell_{3,a}< \ell_{4,a}}\end{footnotesize}}^{n_a}
\frac{\left[{\vZ_{(\vell_1,\vell_2)}}^\top \vT{\vZ_{(\vell_3,\vell_4)}}\right]^2}{4\cdot \prod\nolimits_{i=1}^a 6 \cdot \binom{n_i}{4} }.\]

These are U-statistics based on symmetrized quadratic forms as kernels.

\begin{Le}\label{estimators}
In each of our frameworks  \eqref{asframe1}-\eqref{asframe3} it holds
\begin{itemize}
\item[a)]
$A_1$ is an unbiased estimator for $\tr(\vT\vV_N\vT)$, which is addition dimensional stable and ratio-consistent, if $\prod_{i=1}^a \frac{(n_i-2)!\cdot (n_i-2)!}{n_i!(n_i-4)!}\to 1$ is fulfilled.
\item[b)]
$A_2$ is an unbiased estimator for $\tr\left((\vT\vV_N\vT)^2\right)$, which is addition dimensional stable and ratio-consistent, if $\prod_{i=1}^a \frac{(n_i-4)!\cdot (n_i-4)!}{n_i!(n_i-8)!}\to 1$ is fulfilled.
\end{itemize}
\end{Le}
\begin{re}\phantom{1}\vspace*{-0.5 \baselineskip}
\begin{itemize}
\item[1)]These conditions are fulfilled, for example,  if there  exists $q>0$ with $n_{\min}=\Lan(a^q)$ or similar.
\item[2)] The required number of summations and relating to that, the computation time is increasing really fast for a large number of groups or large sample sizes. From this reason the usage of subsampling versions $A_1^\star(\Upsilon)$ and $A_2^\star(\Upsilon)$  is reasonable, where $\Upsilon\in \N$ denotes the number of summations. As a result, the estimator is not calculated for all possible index combinations but a random subset thereof with size $\Upsilon$. As a result, the computation effort becomes manageable. These estimators and their properties can also be found in the appendix.
\item[3)] In semi-high-dimensional settings, estimating the finite groups' covariance matrices and thus estimating some traces would be attractive. But through the structure of $\vT$, this makes less sense since the other summands of \eqref{EQF} and \eqref{VQF} can not be estimated on their own.
\end{itemize}
\end{re}

 Through the differences of two observations from the same group in defining the random vector $\vZ$, these estimators can be used under both the null hypothesis and the alternative. This is an important property that helps assess the power or similar. Analogical, we could define a vector without these differences. Then the corresponding estimators only work under the null hypothesis, but the amounts of possible combinations are substantially smaller. \\
With these estimators, the estimated version of our test statistic can be formulated by
\[W_N=\frac{Q_N-A_1}{\sqrt{2\cdot A_2}}\ \text{ resp.}\ W_N^\star=\frac{Q_N-A_1^\star(\Upsilon)}{\sqrt{2\cdot A_2^\star(\Upsilon)}} .\]

Through the following Lemma,  the usage of the estimated version instead of the exact one is justified.
\begin{theorem}\label{Theorem4}
Under $\mathcal H_0(T):\vT\vmu=\vnull_{D}$ and one of the 
frameworks \eqref{asframe1}-\eqref{asframe3} the estimated test statistics $W_N$ and $W_N^\star$ have the same asymptotic limit distributions as $\widetilde W_N$, if the respective conditions (a)-(b) from Theorem~\ref{Asymptotik} and the conditions of Lemma~\ref{estimators}  are fulfilled.

\end{theorem}

Unfortunately, in our general setting, it is nearly impossible to estimate $\beta_1$ dimensional stable, which makes the usage of Theorem~\ref{Asymptotik} difficult.
To this aim \cite{pEB} as well as \cite{sattler2018} considered a standardized centred random variable $K_f=\frac{\chi_f^2-f}{\sqrt{2f}}$, with appropriate degrees of freedom. To get a accordance of the first three moments, usually $f_P=\tr^3\left((\vT\vV_N)^2\right)/\tr^2\left((\vT\vV_N)^3\right)$, the so called Pearson-approximation, is used although  also other choices can make sense. In \cite{sattler2018}, the asymptotic distribution of $K_{f_P}$ for $\beta_1\to\{0,1\}$ was investigated. We expand their result in the following theorem

\begin{theorem}\label{Theorem5}
The random variable $K_{{f_P}}$ has, as $N \to \infty$, asymptotically 
\begin{itemize} 
\item[a)]a standard normal distribution  if and only if $\beta_1\to 0$,
\item[b)]a standardized $\left(\chi_1^2-1\right)/\sqrt{2}$ distribution if and only if $\beta_1\to 1$.
\end{itemize}
\vspace{-.5cm}
\end{theorem}
Now, if we found a ratio consistent dimensional stable estimator $\hat f_P$ for $f_P$,  we could use this result to develop asymptotic correct level $\alpha$ tests.
Because an estimator for $\tr\left((\vT\vV_N)^2\right)$ was already developed, especially an estimator for $\tr\left((\vT\vV_N)^3\right)$ is needed.
In \cite{sattler2018}, an appropriate estimator was developed through

\bqan\label{eq:C5}
A_{3}=\hspace*{-0.2cm}\sum\limits_{\begin{footnotesize}\substack{\ell_{1,1},\dots, \ell_{6,1}=1\\
\ell_{1,1}\neq\dots\neq \ell_{6,1}}\end{footnotesize}
}^{n_1}\dots\hspace*{-0.2cm}\sum\limits_{\begin{footnotesize}\substack{\ell_{1,a},\dots, \ell_{6,a}=1\\
\ell_{1,a}\neq\dots\neq \ell_{6,a}}\end{footnotesize}
}^{n_a}
\frac{{\vZ_{(\vell_1,\vell_2)}}^\top \vT{\vZ_{(\vell_3,\vell_4)}}\cdot{\vZ_{(\vell_3,\vell_4)}}^\top \vT{\vZ_{(\vell_5,\vell_6)}}\Cdot {\vZ_{(\vell_5,\vell_6)}}^\top \vT{\vZ_{(\vell_1,\vell_2)}}}{8\cdot \prod\limits_{i=1}^a\frac {  n_i! }{\left(n_i-6\right)!}}.
\eqan

This is an unbiased estimator for $\tr\left((\vT\vV_N)^3\right)$, but the number of summations makes again a subsampling version $A_3^\star(\Upsilon)$ necessary.

Together with the estimators for $\tr(\left(\vT\vV_N\right)^2)$ we can formulate estimators for $f_P$.

\begin{Le}\label{estimators2}
Define estimators for $f_P$ through

\[\hat{f}_P^A= \frac{A_2^3}{A_3^2}\quad \text{ and  }\quad  \hat{f}_P^{A\star}(\Upsilon)= \frac{(A_2^\star(\Upsilon))^3}{(A_3^\star(\Upsilon))^2}.\]
If it exists some $p>1$ which fulfills $\min(n_1,\dots,n_a)=\Lan\left(a^p\right)$, then under  one of the frameworks  it \eqref{asframe1}-\eqref{asframe3} holds
\begin{itemize}
\item[a)] $\frac{1}{\hat{f}_P^A}-\frac{1}{{f}_P}\stackrel{\mathcal{P}}{\to}0.$
\item[b)] $\frac{1}{\hat{f}_P^{A\star}(\Upsilon)}-\frac{1}{{f}_P}\stackrel{\mathcal{P}}{\to}0.$
\item[c)] The results of \Cref{Theorem5} remains valid, if the degrees of freedom $f_P$ are replaced by its estimators $\hat{f}_P$ or  $\hat{f}_P^{A\star}(\Upsilon)$.
\end{itemize}

\end{Le}

With this result, for $\beta_1\to b_1\in\{0,1\}$ we achieve asymptotic correct level $\alpha$ test through $\psi_z^{\star A}=\ind(W_N^{\star A}>z_{1-\alpha})$,   $\psi_{\chi}^{\star A}=\ind(W_N^{\star A}>\chi_{1;1-\alpha}^2)$ and $\varphi_{N}^{\star A}=\ind\{W_N^{\star A}>K_{\hat f_P^{A\star };1-\alpha}\}$. Here, $\chi_{1;1-\alpha}^2$ denotes the $1-\alpha$ quantile of a $\chi_1^2$ distribution and $K_{\hat f_P^{A\star };1-\alpha}$ accordingly the $1-\alpha$ quantile of $K_{\hat f_P^{A\star }}$.\\

Using this approximation makes decision whether $\beta_1\to 0$ or $\beta_1\to 1$ unnecessary. Moreover, from \cite{pEB} and \cite{sattler2018} it is known that it clearly increasing the small sample approximation. Finally through the fact that $f_P\to 1\Leftrightarrow \beta_1\to 1$ and  $f_P\to \infty\Leftrightarrow \beta_1\to 0$, this parameter can be used to determine the behaviour of $\beta_1$, which can not be estimated.

\begin{re}\phantom{1}\vspace*{-0.5 \baselineskip}
\begin{itemize}
\item[1)]
All considered trace estimators had conditions that require a relation between the number of groups and the sample size and therefore are a restriction of our frameworks. So we developed further unbiased dimension stable ratio consistent estimators $B_1(\Upsilon)$, $B_2(\Upsilon)$ and $B_3(\Upsilon)$, without such conditions. Hereby the indices for all groups are the same. Then in the case of unbalanced settings, many observations would not be used, which is why we use a kind of mixing procedure. Because of this fact, they are more complicated and comprehensive and therefore were shifted to the appendix.
\item[2)]
Also for these estimators the necessary number of summations is high, which makes the definition of subsampling versions $B_1^\star(\Upsilon_1,\Upsilon_2), B_2^\star(\Upsilon_1,\Upsilon_2)$ and $B_3^\star(\Upsilon_1,\Upsilon_2)$ reasonable.

\item[3)]
These estimators can be used to replace the previous one on the standardized quadratic form and for the Pearson-approximation. So we get $W_N^B$ resp. $W_N^{B\star}$ and $\hat f_P^{B}$ resp. $\hat f_P^{B\star}(\Upsilon_1,\Upsilon_2)$.
\end{itemize}
\end{re}

\begin{Le}\label{estimators3}
The results of \Cref{Theorem4} and \Cref{estimators2} remains valid even without the conditions on the relations between sample sizes and number of groups,  if we replace the estimators therein with $B_1$, $B_2$ and $B_3$ resp. $B_1^\star(\Upsilon_1,\Upsilon_2), B_2^\star(\Upsilon_1,\Upsilon_2)$ and $B_3^\star(\Upsilon_1,\Upsilon_2)$. 
\end{Le}

Based on this Lemma further asymptotic tests can be defined through $\psi_z^{B\star }=\ind(W_N^{B\star }>z_{1-\alpha})$,   $\psi_{\chi}^{B\star}=\ind(W_N^{B \star}>\chi_{1;1-\alpha}^2)$ and $\varphi_{N}^{B \star}=\ind\{W_N>K_{\hat f_P^{B \star};1-\alpha}\}$.

\section{Simulation}\label{simulations}
In this section, we will focus on framework \eqref{asframe1}, since settings can soon become slightly unwieldy for an increasing number of groups.
We consider two main settings with two groups each time
\begin{itemize}
\item a semi-high-dimensional setting with $d_1=5$ and $d_2=D-d_1$,
\item a high dimensional setting with $d_1=0.2 D$ and $d_2=0.8 D$,
\end{itemize}
 while we have $D=(10,20,40,70,100,150,200,300,400,600,800,1200)$.

 This will be investigated for a small sample size($n_1=10$, $n_2=15$) and medium samples sizes($n_1=20$, $n_2=30$ and $n_1=50$, $n_2=75$) with $\alpha=5\%$.
Since a setting like this is not part of other existing works, no competitors exist to compare. 

The used hypothesis matrix depends on the considered scenario. In scenario $A)$ we have $\beta_1\to 0$  with a hypothesis matrix based on $\vH(A)$ given through
$\vH_{11}(A)=\vP_{d_1,d_1}$, $\vH_{12}(A)=-\vP_{d_1,d_2}$, $\vH_{21}(A)=-\vP_{d_1,d_2}$ and $\vH_{22}(A)=\vP_{d_2,d_2}$. In this scenario we use  two compound symmetry covariance matrices $\vSigma_1=\vI_{d_1} +\vJ_{d_1}/d_1$ and $\vSigma_2=\vI_{d_2} +\vJ_{d_2}/d_2$.
 Here $\vI_d$ denotes the $d\times d$-dimensional identity, while $\vJ_d$ denotes the $d\times d$-dimensional matrix only containing 1's.

For scenario $B)$ based on $\vH(B)$ with
$\vH_{11}(B)=\vJ_{d_1,d_1}$, $\vH_{12}(B)=-\vJ_{d_1,d_2}$, $\vH_{21}(B)=\vJ_{d_1,d_2}$ and $\vH_{22}(B)=\vJ_{d_2,d_2}$, we get $\beta_1 \to 1$.
In both cases we use the unique projection matrix $\vT=\vH^\top (\vH\vH^\top)^+\vH$ based on the Moore-Penrose-Inverse.
Here two different kinds of autoregressive covariance matrices are used with $(\vSigma_1)_{s,t}=0.6^{|s-t|}$ and $(\vSigma_2)_{s,t}=0.6^{|s-t|/(d_2-1)}$.\\\\
 Based on these hypothesis matrices, we consider the type-I error rate of our tests $\psi_z^{B \star}$,   $\psi_{\chi}^{B \star}$ and $\varphi_{N}^{B \star}$. We relinquish here also to constitute the rates of the corresponding estimators $\psi_z^{\star A}$,   $\psi_{\chi}^{\star A}$ and $\varphi_{N}^{\star A}$, since our simulations (see the appendix for more details) show similar results.
 We always use $\Upsilon_2=10$ and $\Upsilon_1$ depending on the estimator and the sample size for the subsampling estimators. For $B_1$ it is $\Upsilon_1=5\cdot N$, while for $B_2$ we chose $\Upsilon_1=10\cdot N$ and for $B_3$ finally $\Upsilon_1=100\cdot N$. This choice bears that the estimators have different orders of the kernel of the U-statistics and, therefore, significant differences in the number of possible index combinations. \\

\begin{figure}[htbp]
\setlength{\abovecaptionskip}{5pt} 
\setlength{\belowcaptionskip}{1pt} 
 
\includegraphics[clip,trim=1mm 1mm 1mm 3mm,  width=\textwidth]{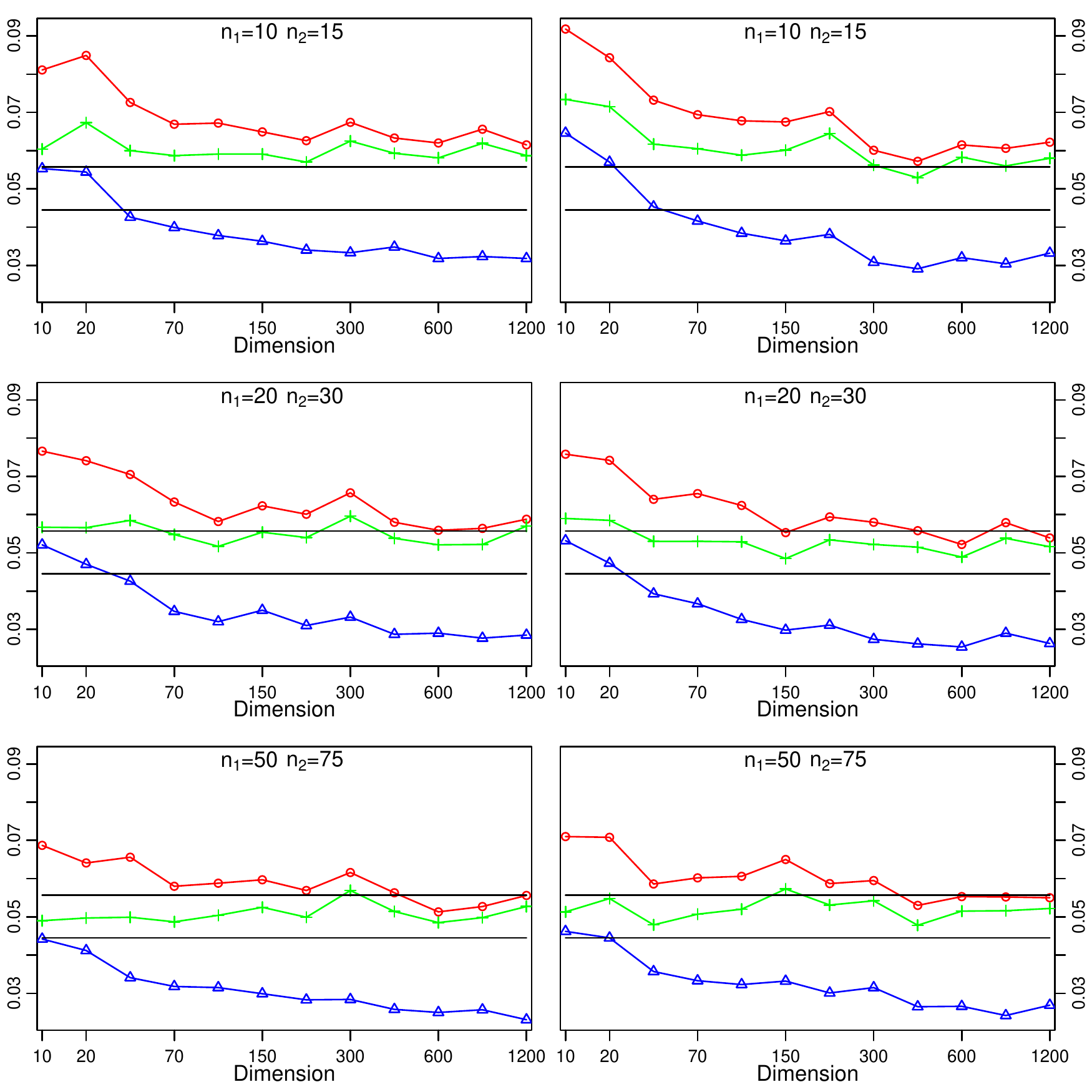} 

\caption{Simulated type I-error rates ($\alpha=5\%$) for the statistic $W_N^B$ compared with the critical values of a standard normal(\textcolor{red}{$\circ $}), standardized $\chi_1^2$(\textcolor{blue}{$\bigtriangleup$}) and $K_f$-distribution(\textcolor{green}{$+$}) under the null hypothesis $B)$ for increasing dimension $D$ with covariance matrices $(\vSigma_1)_{s,t}=0.6^{|s-t|}$ and $(\vSigma_2)_{s,t}=0.6^{|s-t|/(d_2-1)}$.
On the left side it is a semi-high-dimensional setting with  $d_1=5$ and $d_2=D-5$ while on the right side it holds $d_1=0.2D$ and $d_2=0.8D$.}
\label{BildSimulation1} 

\end{figure} 
In \Cref{BildSimulation1} we see that
  $\psi_{z}^{B \star}$ for all sample sizes gets closer to the $\alpha$ level with increasing dimension but needs $D\geq 300$ and $N\geq 125$ to be in the $99\%$ binomial interval, independent of the setting.
  Conversely, $\phi_{N}^{B \star}$ has already for $N>25$ a type-I-error rate which is almost entirely in the binomial interval even for very small dimension $D$. This shows again that the using $K_{\widehat f_P}$ leads to a better and faster approximation. A mistakenly usage of $\psi_{\chi}^{B \star}$ would lead to a way to conservative test for increasing dimension, with error rates about $3\%$.
In scenario A) our tests perform slightly better for the semi-high dimensional setting, especially for smaller dimensions, than for the pure high-dimensional setting. This is surprising since, for the semi-high-dimensional setting, the dimensions are more unbalanced, which could be expected to be more challenging.
\begin{figure}[htbp]
\setlength{\abovecaptionskip}{5pt} 
\setlength{\belowcaptionskip}{1pt} 
 
\includegraphics[clip,trim=1mm 1mm 1mm 3mm,  width=\textwidth]{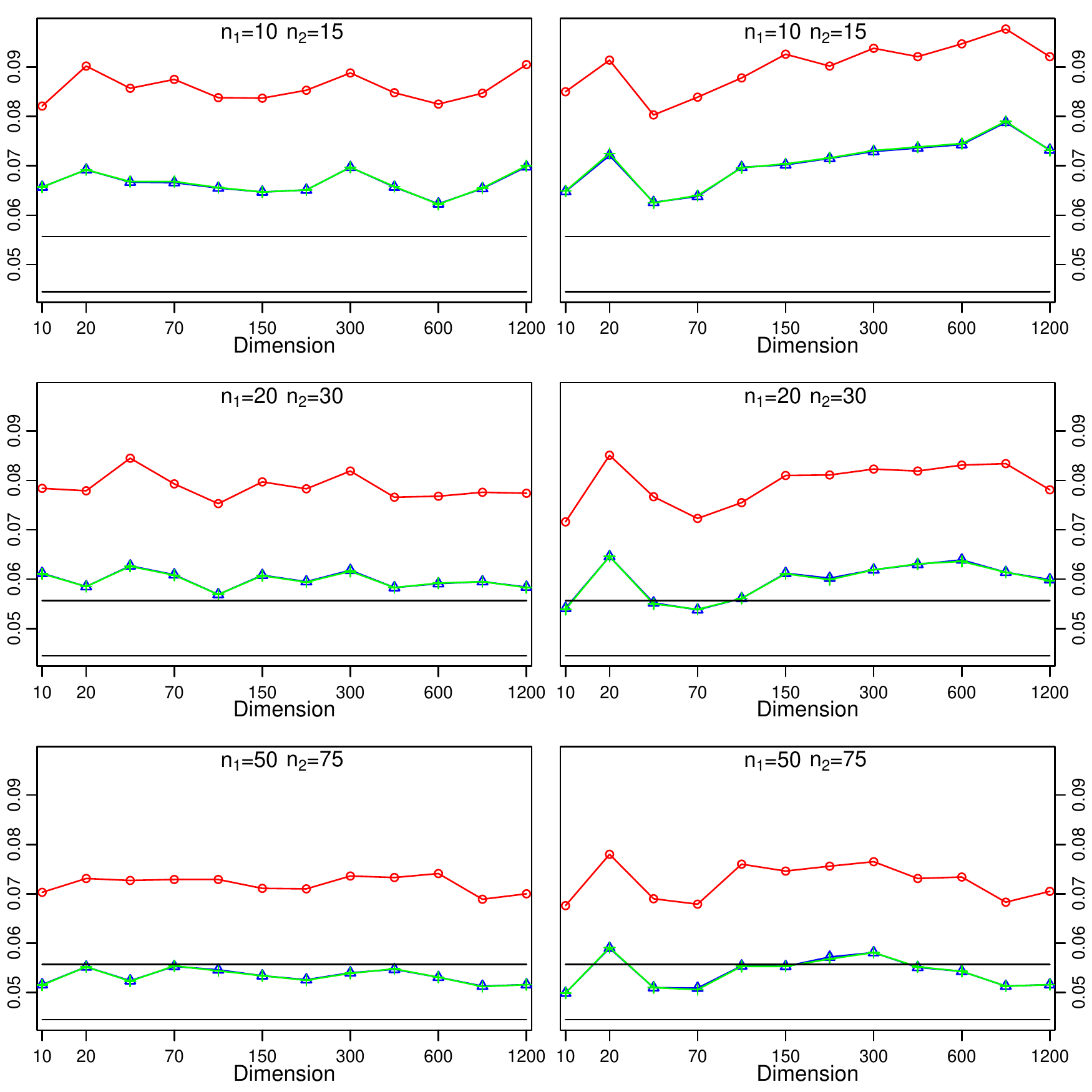} 

\caption{Simulated type I-error rates ($\alpha=5\%$) for the statistic $W_N^B$ compared with the critical values of a standard normal(\textcolor{red}{$\circ $}), standardized $\chi_1^2$(\textcolor{blue}{$\bigtriangleup$}) and $K_f$-distribution(\textcolor{green}{$+$}) under the null hypothesis $A)$ for increasing dimension $D$ with covariance matrices $\vSigma_1=\vP_{d_1}+\vJ_{d_1}/{d_1}$ and $\vSigma_2=\vP_{d_2}+\vJ_{d_2}/{d_2}$.
On the left side it is a semi-high-dimensional setting with  $d_1=5$ and $d_2=D-5$ while on the right side it holds $d_1=0.2D$ and $d_2=0.8D$.}
\label{BildSimulation2} 
\end{figure} 

For scenario B), we also see in \Cref{BildSimulation2} that the type-I-error rate is better in the semi-high-dimensional setting. Here a mistakenly usage of $\psi_{z}^{B \star}$ would lead to a liberal test with error rates between 0.065 and 0.1 . Because of the fast convergence of $f_P \to 1$, both tests based on the standardized $\chi^2$ distribution perform identical and need sample size $N>50$ to be continuously in the binomial interval. Moreover, for the high-dimensional setting, it is remarkable that the type-I error rate of $\psi_{z}^{B \star}$ and $\psi_{\chi}^{B \star}$ seems only to be shifted. For the semi-high-dimensional setting, the rates of both tests have a similar shape but also have apparent differences.\\

All in all,   $\phi_{N}^{B \star}$ showed good performance in both settings and for comparable small dimensions and sample sizes, while for larger sample sizes, the used trace estimators are closer to the unknown values. This can be assumed to be the main reason for the better performance for medium sample sizes.

\section{Conclusion}
The present paper introduces an approach for split-plot designs, where the single groups have different dimensions. These kinds of settings can occur for various reasons but were, so far, rarely considered in particular in the context of high-dimensional data. Therefore we expand the approach of \cite{sattler2018} for groups with different dimensions, which includes especially semi-high-dimensional settings.
With this, we derived the asymptotic distribution of our test statistic under different conditions as well as new and innovative estimators for the required traces. For practical usage of this approach, an alternative approach for determining critical values and different kinds of subsampling estimators were developed. An extensive simulation study was done to investigate the properties of this new approach in different settings. Here for different kinds of partially very challenging settings, the type-I-error rate as central property shows the good performance of our approach already for medium sample sizes. Here the different kinds of estimators performed comparably well, allowing us to choose them appropriately.

The next step in expanding this approach would be to allow different dimensions within the groups for future research. In this way, it would be possible to investigate high-dimensional clustered data sets. With this estimating the necessary traces without considerable model restrictions will be challenging.
Analogically, an extension of the presented approach for data with missing values would be desirable, but again need, new concepts to estimate the necessary traces containing the unknown covariance matrices.

Also, reformulating other central approaches regarding high-dimensional data like the work from \cite{chen2010} would be of great interest.

\section*{Acknowledgement}
 This work was supported by the German Research Foundation project DFG-PA2409/3-2.
\newpage
\section{Appendix}
\subsection{Further Simulations}
Since we introduced two different kinds of estimators, we want to investigate the influence of the chosen estimator on the type-I error rate. Thereto, we repeat the type-I error investigation from above now for $W_N^A$. In \Cref{BildSimulation3} and \Cref{BildSimulation4} we can see that the performance is very similar for both hypotheses and both settings

\begin{figure}[H]
\setlength{\abovecaptionskip}{5pt} 
\setlength{\belowcaptionskip}{1pt} 
 
\includegraphics[clip,trim=1mm 1mm 1mm 3mm,  width=\textwidth]{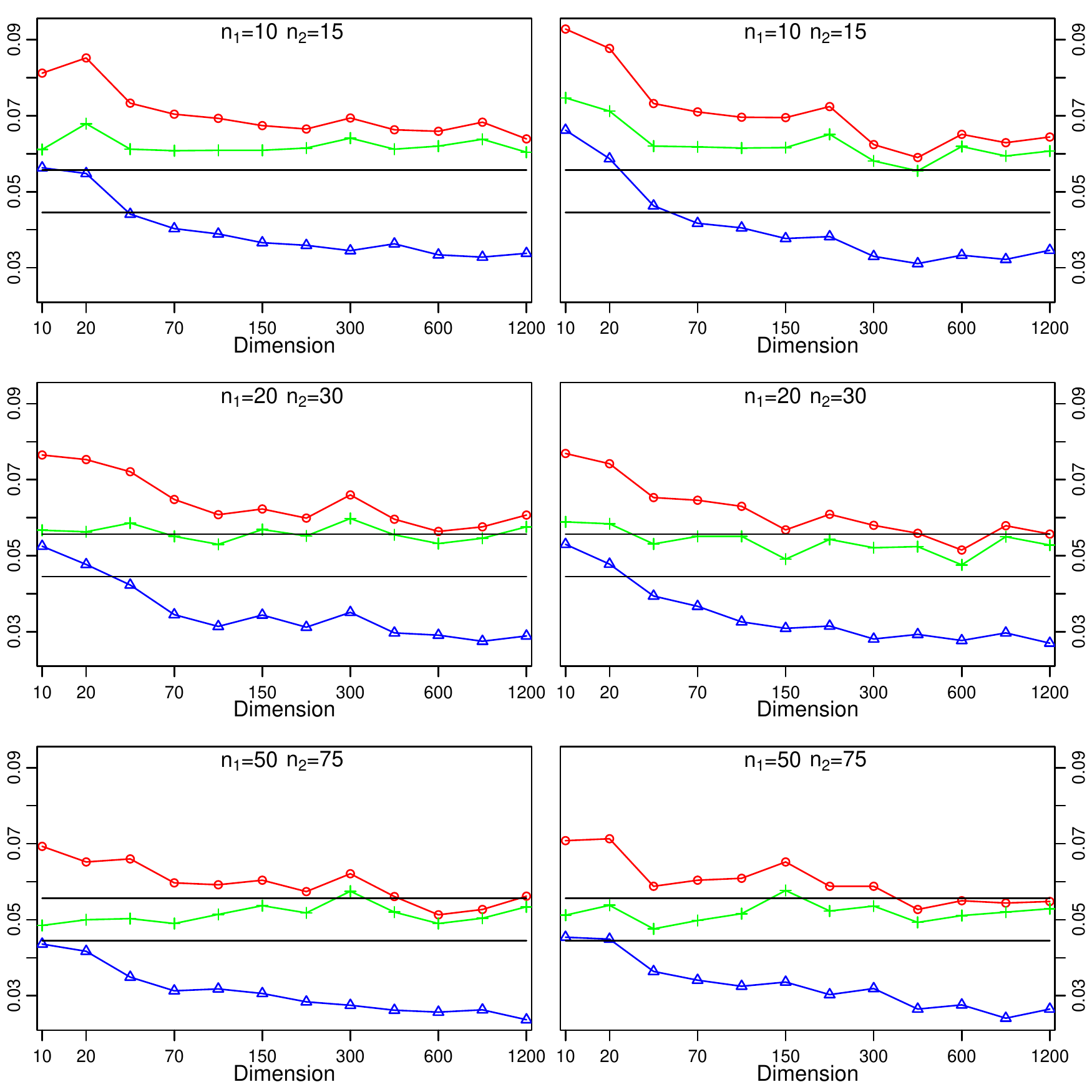} 

\caption{Simulated type I-error rates ($\alpha=5\%$) for the statistic $W_N^A$ compared with the critical values of a standard normal(\textcolor{red}{$\circ $}), standardized $\chi_1^2$(\textcolor{blue}{$\bigtriangleup$}) and $K_f$-distribution(\textcolor{green}{$+$}) under the null hypothesis $B)$ for increasing dimension $D$ with covariance matrices $(\vSigma_1)_{s,t}=0.6^{|s-t|}$ and $(\vSigma_2)_{s,t}=0.6^{|s-t|/(d_2-1)}$.
On the left side it is a semi-high-dimensional setting with  $d_1=5$ and $d_2=D-5$ while on the right side it holds $d_1=0.2D$ and $d_2=0.8D$.}
\label{BildSimulation3} 
\end{figure} 

\begin{figure}[H]
\setlength{\abovecaptionskip}{5pt} 
\setlength{\belowcaptionskip}{1pt} 
 
\includegraphics[clip,trim=1mm 1mm 1mm 3mm,  width=\textwidth]{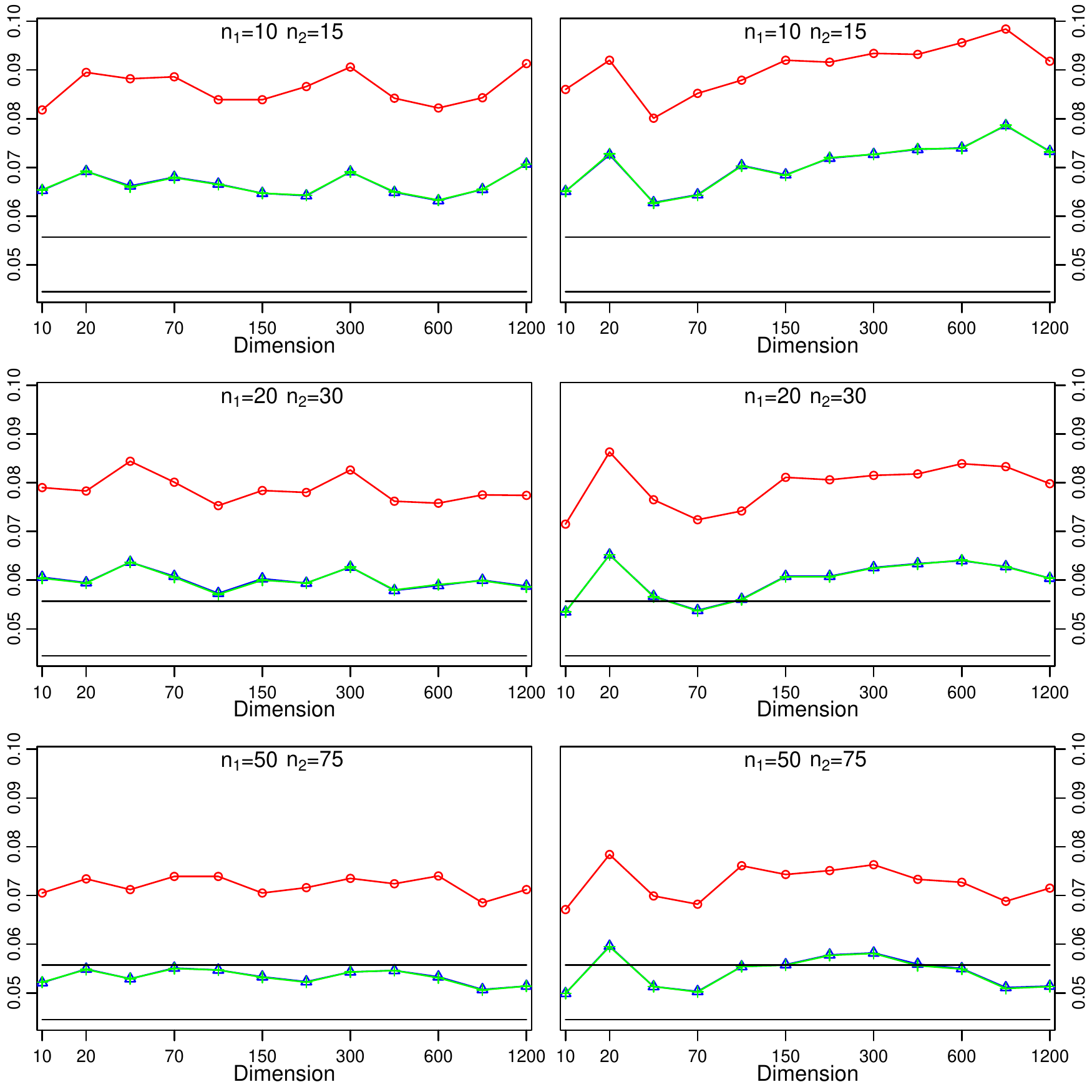} 

\caption{Simulated type I-error rates ($\alpha=5\%$) for the statistic $W_N^A$ compared with the critical values of a standard normal(\textcolor{red}{$\circ $}), standardized $\chi_1^2$(\textcolor{blue}{$\bigtriangleup$}) and $K_f$-distribution(\textcolor{green}{$+$}) under the null hypothesis $A)$ for increasing dimension $D$ with covariance matrices $\vSigma_1=\vP_{d_1}+\vJ_{d_1}/{d_1}$ and $\vSigma_2=\vP_{d_2}+\vJ_{d_2}/{d_2}$.
On the left side it is a semi-high-dimensional setting with  $d_1=5$ and $d_2=D-5$ while on the right side it holds $d_1=0.2D$ and $d_2=0.8D$.
} 
\label{BildSimulation4} 
\end{figure}

Moreover, if the estimators were implemented efficiently, the computation times of both kinds of estimators are comparable. So because of the necessary restrictions of  $A_1$, $A_2$ and $A_3$ it is recommendable to use $B_1$, $B_2$ and $B_3$ although their construction is more complicated.
 Even for the comparable high values of $\Upsilon_1$ it was up to 15 times faster. Together with the lesser requirements, therefore, these estimators are preferable, although they are more complicated.

\subsection{Proofs}
The new estimators $A_1$ and $A_2$ can be seen as a combination of the estimators $A_{i,1}$ and $A_{i,3}$ from \cite{sattler2018}  with the pooled vector $\vZ$.

\begin{proof}[Proof of Lemma \ref{estimators}]\phantom{1}\vspace*{-0.25 \baselineskip}

With the formulas for moments of the quadratic forms from \cite{sattler2018} and an adaptation of estimators therein, we calculate\\

$\begin{array}{ll}
\E\left({A_1}\right)&=\frac{1}{2\cdot \prod\limits_{i=1}^a\binom{  n_i }{2 }}\sum\limits_{\begin{footnotesize}\substack{\ell_{1,1}, \ell_{2,1}=1\\\ell_{1,1}<\ell_{2,1}}\end{footnotesize}}^{n_1}
\dots\sum\limits_{\begin{footnotesize}\substack{\ell_{1,a}, \ell_{2,a}=1\\\ell_{1,a}<\ell_{2,a}}\end{footnotesize}}^{n_a}
\E\left({\vZ_{(\vell_1,\vell_2)}}^\top \vT{\vZ_{(\vell_1,\vell_2)}}\right)
\\[3.2ex]
&=\frac{1}{2\cdot \prod\limits_{i=1}^a\binom{  n_i }{2} }\sum\limits_{\begin{footnotesize}\substack{\ell_{1,1}, \ell_{2,1}=1\\\ell_{1,1}< \ell_{2,1}}\end{footnotesize}}^{n_1}
\dots\sum\limits_{\begin{footnotesize}\substack{\ell_{1,a}, \ell_{2,a}=1\\\ell_{1,a}< \ell_{2,a}}\end{footnotesize}}^{n_a}
\tr\left(2\vT\vV_N\vT\right)\\
&=\tr\left(\vT\vV_N\vT\right)
\end{array}$\\\\
and\\

$\begin{array}{ll}
&\Var\left({A_1}\right)\\=&\sum\limits_{\begin{footnotesize}\substack{\ell_{1,1}, \ell_{2,1}=1\\\ell_{1,1}<\ell_{2,1}}\end{footnotesize}}^{n_1}
\dots\sum\limits_{\begin{footnotesize}\substack{\ell_{1,a}, \ell_{2,a}=1\\\ell_{1,a}<\ell_{2,a}}\end{footnotesize}}^{n_a}
\sum\limits_{\begin{footnotesize}\substack{\ell_{1,1}', \ell_{2,1}'=1\\\ell_{1,1}'<\ell_{2,1}'}\end{footnotesize}}^{n_1}
\dots\sum\limits_{\begin{footnotesize}\substack{\ell_{1,a}', \ell_{2,a}'=1\\\ell_{1,a}'< \ell_{2,a}'}\end{footnotesize}}^{n_a}
\frac{\Cov\left({\vZ_{(\vell_1,\vell_2)}}^\top \vT{\vZ_{(\vell_1,\vell_2)}},{\vZ_{(\vell_1',\vell_2')}}^\top \vT{\vZ_{(\vell_1',\vell_2')}}\right)}{\left(2\cdot \prod\limits_{i=1}^a\binom {  n_i }{2} \right)^2}
\\[2.5ex]
\leq& \frac{\prod\limits_{i=1}^a\binom {  n_i }{2}-\prod\limits_{i=1}^a\binom {  n_i-2 }{2}}
{2\cdot \prod\limits_{i=1}^a\binom {  n_i }{2} } \Var\left({\vZ_{(\vell_1,\vell_2)}}^\top \vT{\vZ_{(\vell_1,\vell_2)}}\right)
\\[2.2ex]
=&\frac{\prod\limits_{i=1}^a\binom {  n_i }{2}-\prod\limits_{i=1}^a\binom {  n_i-2 }{2}}
{2\cdot \prod\limits_{i=1}^a\binom {  n_i }{2} }  \cdot \Lan\left(\tr^2\left(\vT\vV_N\right)\right).
\end{array}$\\\\\\
Similar we get\\\\
$\begin{array}{ll}
\E\left({A_2}\right)&=\frac{1}{4\cdot \prod\limits_{i=1}^a 6\cdot\binom{  n_i }{4 }}\sum\limits_{\begin{footnotesize}\substack{\ell_{1,1}\neq \dots \neq  \ell_{4,1}=1\\\ell_{1,1}<\ell_{2,1},\dots,\ell_{3,1}< \ell_{4,1}}\end{footnotesize}}^{n_1}
\dots\sum\limits_{\begin{footnotesize}\substack{\ell_{1,a}\neq\dots\neq \ell_{4,a}=1\\\ell_{1,a}<\ell_{2,a},\dots,\ell_{3,a}< \ell_{4,a}}\end{footnotesize}}^{n_a}
\E\left(\left[{\vZ_{(\vell_1,\vell_2)}}^\top \vT{\vZ_{(\vell_3,\vell_4)}}\right]^2\right)
\\[3ex]
&=\frac{1}{4\cdot \prod\limits_{i=1}^a 6\cdot\binom{  n_i }{4 }}\sum\limits_{\begin{footnotesize}\substack{\ell_{1,1}\neq \dots \neq  \ell_{4,1}=1\\\ell_{1,1}<\ell_{2,1},\dots,\ell_{3,1}< \ell_{4,1}}\end{footnotesize}}^{n_1}
\dots\sum\limits_{\begin{footnotesize}\substack{\ell_{1,a}\neq\dots\neq \ell_{4,a}=1\\\ell_{1,a}<\ell_{2,a},\dots,\ell_{3,a}< \ell_{4,a}}\end{footnotesize}}^{n_a}
\tr\left(\left(2\vT\vV_N\vT\right)^2\right)\\
&=\tr\left(\left(\vT\vV_N\vT\right)^2\right)
\end{array}$\\\\ and\\

$\begin{array}{ll}
&\Var\left({A_2}\right)\\=&\hspace*{-0.3cm}\sum\limits_{\begin{footnotesize}\substack{\ell_{1,1}\neq \dots \neq  \ell_{4,1}=1\\\ell_{1,1}<\ell_{2,1},\ell_{3,1}< \ell_{4,1}}\end{footnotesize}}^{n_1}
\hspace*{-0.2cm}\dots\hspace*{-0.2cm}\sum\limits_{\begin{footnotesize}\substack{\ell_{1,a}\neq\dots\neq \ell_{4,a}=1\\\ell_{1,a}<\ell_{2,a},\ell_{3,a}< \ell_{4,a}}\end{footnotesize}}^{n_a}
\sum\limits_{\begin{footnotesize}\substack{\ell_{1,1}'\neq \dots \neq  \ell_{4,1}'=1\\\ell_{1,1}'<\ell_{2,1}',\ell_{3,1}'< \ell_{4,1}'}\end{footnotesize}}^{n_1}
\hspace*{-0.2cm}\dots\hspace*{-0.1cm}\sum\limits_{\begin{footnotesize}\substack{\ell_{1,a}'\neq\dots\neq \ell_{4,a}'=1\\\ell_{1,a}'<\ell_{2,a}',\ell_{3,a}'< \ell_{4,a}'}\end{footnotesize}}^{n_a}\hspace*{-0.3cm}
\frac{\Cov\left(\left[{\vZ_{(\vell_1,\vell_2)}}^\top \vT{\vZ_{(\vell_3,\vell_4)}}\right]^2,\left[{\vZ_{(\vell_1',\vell_2')}}^\top \vT{\vZ_{(\vell_3',\vell_4')}}\right]^2\right)}{\left(4\cdot \prod\limits_{i=1}^a 6 \Cdot\binom {  n_i!}{4} \right)^2}
\\[2.5ex]
\leq & \frac{\prod\limits_{i=1}^a 6\cdot \binom {  n_i }{4}-\prod\limits_{i=1}^a 6 \cdot \binom {   n_i-4 }{4}}
{4\cdot \prod\limits_{i=1}^a6\Cdot \binom {  n_i }{4} } \Var\left(\left[{\vZ_{(\vell_1,\vell_2)}}^\top \vT{\vZ_{(\vell_3,\vell_4)}}\right]^2\right)
\\[2.5ex]
=& \frac{\prod\limits_{i=1}^a 6\cdot \binom {  n_i }{4}-\prod\limits_{i=1}^a 6 \cdot \binom {   n_i-4 }{4}}
{4\cdot \prod\limits_{i=1}^a6\Cdot \binom {  n_i }{4} }  \cdot \Lan\left(\tr^2\left(\left(\vT\vV_N\right)^2\right)\right).
\end{array}$\\\\
Therefore, with the condition on their respective product, both estimators are unbiased, dimensional stable and ratio-consistent.
\end{proof}

Instead of using each possible index combination a subsampling version is based on a random subsample. Thereto, for each $i=1,\dots,a$ and $\upsilon=1,\dots,\Upsilon$  independently drawn random subsamples $\{\sigma_{1i}(\upsilon),\dots,\sigma_{4i}(\upsilon)\}$ of length $4$ from $\{1,\dots,n_i\}$ are merged in a joint random vector $\vsigma(\upsilon,4) =(\vsigma_1(\upsilon,4),...,\vsigma_1(\upsilon,4))= (\sigma_{11}(\upsilon),\dots,\sigma_{1a}(\upsilon),\dots,\sigma_{4a}(\upsilon))$. Similar we can define $\vsigma(\upsilon,2)$ and $\vsigma(\upsilon,6)$.

With these vectors, the subsampling version of our estimator for $\Upsilon$ repetitions is given through

\[A_1^\star(\Upsilon):=\frac{1}{2\Upsilon}\sum\limits_{\upsilon=1}^\Upsilon {\vZ_{(\vsigma_1(\upsilon,2),\vsigma_2(\upsilon,2))}}^\top \vT{\vZ_{(\vsigma_1(\upsilon,2),\vsigma_2(\upsilon,2))}}\]
and 
\[A_2^\star(\Upsilon):=\frac{1}{4\Upsilon}\sum\limits_{\upsilon=1}^\Upsilon  \left[{\vZ_{(\vsigma_1(\upsilon,2),\vsigma_2(\upsilon,2))}}^\top \vT{\vZ_{(\vsigma_3(\upsilon,2),\vsigma_4(\upsilon,2))}}\right]^2.\]

\begin{Le}\label{estimatorsub}
In each of the frameworks  \eqref{asframe1}-\eqref{asframe3}
\begin{itemize}
\item[a)]$A_1^\star(\Upsilon)$ is an unbiased estimator for $\tr(\vT\vV_N\vT)$. If $\prod_{i=1}^a \frac{(n_i-2)!\cdot (n_i-2)!}{n_i!(n_i-4)!}\to 1$ holds, it is additional dimensional stable and ratio-consistent, for $\Upsilon\to \infty$.
\item[b)]
$A_2^\star(\Upsilon)$ is an unbiased estimator for $\tr\left((\vT\vV_N\vT)^2\right)$. If $\prod_{i=1}^a \frac{(n_i-4)!\cdot (n_i-4)!}{n_i!(n_i-8)!}\to 1$ holds, it is additional dimensional stable and ratio-consistent, for $\Upsilon\to \infty$.
\end{itemize}
\end{Le}

\begin{proof}

For the subsampling version, we take the same steps as for the comparable estimators from \cite{sattler2018} and use some results shown therein.
Denote with $\F(\sigma_i(\Upsilon,m))$  the smallest $\sigma$-field, which contains $\sigma_i(\upsilon, m)$  $\forall \upsilon\in \Upsilon$.
Then we get

$\begin{array}{ll}
\E\left({A_1^\star}(\Upsilon)\right)

&=\frac{1}{2\Upsilon}\sum\limits_{\upsilon=1}^\Upsilon\E\left(  {\vZ_{(\vsigma_1(\upsilon,2),\vsigma_2(\upsilon,2))}}^\top \vT{\vZ_{(\vsigma_1(\upsilon,2),\vsigma_2(\upsilon,2))}}\right)
\\[1.8ex]
&=\frac{1}{2\Upsilon}\sum\limits_{\upsilon=1}^\Upsilon\E\left(\Lambda_{4}(\ell_{1,1},\dots,\ell_{2,a})\right)
\\[1.8ex]

&{=}\frac{1}{2\Upsilon}\sum\limits_{\upsilon=1}^\Upsilon
\tr\left(2\vT\vV_N\right)=\tr\left(\vT\vV_N\right)
\end{array}$\\\\
and \\\\
$\begin{array}{l}
\Var\left(\E\left({A_1^\star}(\Upsilon)|\F(\vsigma(\Upsilon,2))\right)\right)
=\Var\left(\tr\left(\left(\vT\vV_N\right)\right)\right)=0.
\end{array}$\\\\\\
With this and $M(\Upsilon,\vsigma(\upsilon,2))$ as the notation  of  the amount of pairs $(k,\ell)\in \N_\Upsilon\times \N_\Upsilon$, which fulfill  $\vsigma(k,2)$ and $\vsigma(\ell,2)$ have totally different elements, we get

$\begin{array}{ll}
&\Var\left({A_1^\star}(\Upsilon)\right)
\\[1.0ex]
=&0+\E\left(\Var\left({A_1^\star}(\Upsilon)|\F(\vsigma(\Upsilon,2))\right)\right)
\\[1.0ex]
{\leq}&\frac{1}{4 \Upsilon^2}\E\left(\sum\limits_{(j,\ell)\in \N_\Upsilon\times \N_\Upsilon \setminus M(\Upsilon,\vsigma(\upsilon,2))} \Var\left( {\vZ_{(\vsigma_1(j,2),\vsigma_2(j,2))}}^\top \vT{\vZ_{(\vsigma_1(j,2),\vsigma_2(j,2))}}|\F(\vsigma(\Upsilon,2))\right)\right)
\\[2.4ex]
=&\frac{\E\left(|\N_\Upsilon\times \N_\Upsilon \setminus M(\Upsilon,\vsigma(\upsilon,2))|\right)}{\Upsilon^2}\cdot \frac{\Var\left({\vZ_{(\vell_1,\vell_2)}}^\top \vT\vZ_{(\vell_1,\vell_2)}\right)}{4}
\\[1.0ex]
{\leq}&
\left(1-\left(1-\frac{1}{\Upsilon}\right)\cdot \prod\limits_{i=1}^a\frac{ \binom{n_i-2} {2}}{\binom{n_i} {2}}\right)\cdot \tr^2\left(\left(\vT\vV_N\right)\right).
\end{array}$\\\\\\
Hereby, the upper bond for  $\E\left(|\N_\Upsilon\times \N_\Upsilon \setminus M(\Upsilon,\vsigma(\upsilon,2))|\right)$ was proven in \cite{sattler2018}.
The same calculations  with $M(\Upsilon,\vsigma(\upsilon,4))$ for our other estimator lead to

$\begin{array}{ll}
\E\left({A_2^\star}(\Upsilon)\right)

&=\frac{1}{4\Upsilon}\sum\limits_{\upsilon=1}^\Upsilon\E\left(\left[{\vZ_{(\vsigma_1(\upsilon,2),\vsigma_2(\upsilon,2))}}^\top \vT{\vZ_{(\vsigma_3(\upsilon,2),\vsigma_4(\upsilon,2))}}\right]^2 \right) 
\\[1.8ex]
&=\frac{1}{4\Upsilon}\sum\limits_{\upsilon=1}^\Upsilon\E\left({\vZ_{(\vell_1,\vell_2)}}^\top \vT\vZ_{(\vell_3,\vell_4)}\right).
\\[1.8ex]

&{=}\frac{1}{4\Upsilon}\sum\limits_{\upsilon=1}^\Upsilon
\tr\left(\left(2\vT\vV_N\right)^2\right)=\tr\left(\left(\vT\vV_N\right)^2\right)
\end{array}$\\\\\\
as well as\\\\
$\begin{array}{l}
\Var\left(\E\left({A_2^\star}(\Upsilon)|\F(\vsigma(\Upsilon,4))\right)\right)
=\Var\left(\tr\left(\left(\vT\vV_N\right)^2\right)\right)=0
\end{array}$\\\\\\ and finally to\\\\
$\begin{array}{ll}
&\Var\left({A_2^\star}(\Upsilon)\right)
\\[1.0ex]
=&0+\E\left(\Var\left({A_2^\star}(\Upsilon)|\F(\vsigma(\Upsilon,4))\right)\right)
\\[1.0ex]&
{\leq}\frac{1}{16 \Upsilon^2}\E\left(\sum\limits_{(j,\ell)\in \N_\Upsilon\times \N_\Upsilon \setminus M(\Upsilon,\vsigma(b,4))} \Var\left(\left[{\vZ_{(\vsigma_1(\upsilon,2),\vsigma_2(\upsilon,2))}}^\top \vT{\vZ_{(\vsigma_3(\upsilon,2),\vsigma_4(\upsilon,2))}}\right]^2\right)\right)
\\[2.4ex]
=&\frac{\E\left(|\N_\Upsilon\times \N_\Upsilon \setminus M(\Upsilon,\vsigma(\upsilon,4))|\right)}{\Upsilon^2}\cdot \frac{\Var\left(\left[{\vZ_{(\vell_1,\vell_2)}}^\top \vT\vZ_{(\vell_3,\vell_4)}\right]^2\right)}{16}
\\[1.0ex]
{\leq}&
\left(1-\left(1-\frac{1}{\Upsilon}\right)\cdot \prod\limits_{i=1}^a\frac{ \binom{n_i-4} {4}}{\binom{n_i} {4}}\right)\cdot  \Lan\left(\tr^2\left(\left(\vT\vV_N\right)^2\right)\right).
\end{array}$\\\\

These results show that the estimators are unbiased and dimensional-stable if 
\[\left(1-\left(1-\frac{1}{\Upsilon}\right)\cdot \prod\limits_{i=1}^a\frac{ \binom{n_i-2} {2}}{\binom{n_i} {2}}\right)\ \text{ resp. }\ \left(1-\left(1-\frac{1}{\Upsilon}\right)\cdot \prod\limits_{i=1}^a\frac{ \binom{n_i-4} {4}}{\binom{n_i} {4}}\right)\]\\
goes asymptotically to zero. Hence, it is necessary that $\Upsilon\to \infty$ as well as the second part of the respective product goes to 1. While the first point is easy to fullfil, the second one is done through requirements on the relation between samples sizes and number of groups.

\end{proof}

Reasonable choices for $\Upsilon$ are functions applied on $N$, in particular first-grade polynomials. Moreover, it seems obvious that for $B_2$, a higher value $\Upsilon$ should be used than for $B_1$. This would take into account that the respective amount of possible index combinations is essentially greater for $B_2$.

\begin{proof}[Proof of Theorem \ref{Theorem4}]

Based on the ratio-consistency of the estimators, we calculate for the standardized quadratic form\\

$\begin{array}{ll}
W_N&
=\left(\frac{Q_N-\tr\left(\vT\vV_N\right)}
{\sqrt{2\tr\left(\left(\vT\vV_N\right)^2\right)}}-\frac{A_{1}-\tr\left(\vT\vV_N\right)}
{\sqrt{2\tr\left(\left(\vT\vV_N\right)^2\right)}}\right)\cdot
\sqrt{\frac{\tr\left(\left(\vT\vV_N\right)^2\right)}{A_2}}
\\[2.5ex]
&=\left(\frac{Q_N-\tr\left(\vT\vV_N\right)}
{\sqrt{2\tr\left(\left(\vT\vV_N\right)^2\right)}}-\lan_P(1)\right)\cdot
(1+\lan_P(1))
\\[2.5 ex]
&=\widetilde W_N+\widetilde W_N\cdot \lan_P(1)-\lan_P(1)-\lan_P(1)\cdot \lan_P(1).
\end{array}$\\\\

If one of the conditions of \Cref{Asymptotik} is fulfilled, with Slutzky $\widetilde W_N\cdot \lan_P(1)$ converges in distribution to zero. Thereby, since the last two parts converge in probability to zero, and therefore also in distribution, with Slutzky's theorem, the asymptotical distributions of $\widetilde W_N$ and $ W_N$ coincides.\\\\

Based on \Cref{estimatorsub} the subsampling estimators $A_1^\star(\Upsilon)$ and $A_2^\star(\Upsilon)$ have the same properties like $A_1$ and $A_2$. Therefore the result for $W_N^\star$ follows analogously.

\end{proof}
\begin{proof}[Proof of Theorem \ref{Theorem5}]
The theorem is an extension of Theorem 4.1 from \cite{sattler2018}. Since they never used $\vT\vmu=\vnull$ for their proof, it remains to prove the "if and only if" part.
\begin{itemize}
\item[a)]

Since $V_N$ and all eigenvalues depend on $N$, we will denote this in this proof by $\beta_1(N)$ and similar $f_P(N)$. 
Assume $\beta_1(N)\nrightarrow 0$. From the Bolzano-Weierstrass theorem we know that there exists a subsequence $N'$ with $1/f_P(N')\to \gamma$ and $\beta_1(N')\to b_1\neq 0$.
With Lemma A.8 from \cite{sattler2018} it follows that $\gamma\neq 0$ and therefore $f_P(N')\to 1/\gamma\in [1,\infty)$.
Then we get for $Z\sim\mathcal{N}(0,1)$
\[\begin{array}{ll}
&K_{f_P(N')}\stackrel{\mathcal D}{\to} Z\\[1.2ex]
\Leftrightarrow& \frac{\chi_{1/\gamma}^2- 1/\gamma}{\sqrt{2/\gamma}} \stackrel{\mathcal D}{=} Z\\[1.2ex]
\Leftrightarrow& \chi_{1/\gamma}^2 \stackrel{\mathcal D}{=} {\sqrt{2/\gamma}}\cdot Z+1/\gamma.

\end{array}\]

Now on the right side it is a normal distribution, while on the left side it is a $\chi^2$- distribution with $\gamma^{-1}$ degrees of freedom. Therefore with $C_1,C_2\stackrel{i.i.d.}{\sim}\Gamma(\gamma/4,2)$ we can rewrite it through

\[ C_1+C_2\stackrel{\mathcal D}{=} {\sqrt{2/\gamma}}Z+1/\gamma.\]
With Cram{\'e}r's theorem from \cite{cramer} this is a contradiction, and therefore no $b_1\neq 0$ can exist. So all convergent subsequences of $\beta_1(N)$ goes to $0$, and therefore it holds $\beta_1\to 0$.

\item[b)]
 Again with the Bolzano-Weierstrass theorem we know the existence of at least one convergent subsequence for $1/f_P$ and $\beta_1$. So for an arbitrary of these  subsequences $N'$ we know $1/f_P(N')\to \gamma\in(0,1]$ and $\beta_1(N')\to b_1\in(0,1]$. From the rules for characteristic functions we get
 \[\begin{array}{ll}
 &K_{f_P(N')}\stackrel{\mathcal D}{=}  (\chi_1^2-1)/\sqrt{2}\\[1.2ex]
 \Leftrightarrow &|\varphi_{K_{f_P(N')}}(s)|\to \frac{1}{1+2s^2}\\[1.2ex]
\Leftrightarrow  &\frac{f_P(N')^3}{f_P(N')^3+2s^2}\to \frac{1}{1+2s^2}.

\end{array}\]

With $f_P(N')\to 1/\gamma$  and $s=1$ this means,
\[\frac{1/\gamma^3}{1/\gamma^3+2}=\frac{1}{3},\]
which only solution is $\gamma=1$ and therefore $b_1=1$. So all convergent subsequences have the same limit and therefore we knew $\beta_1\to 1$.
\end{itemize}
\end{proof}

\begin{proof}[Proof of Lemma \ref{estimators2}]

This estimator $A_3$ was already used in \cite{sattler2018} and therein
\[\begin{array}{l}
 \E\left(A_{3}\right)=\tr\left(\left(\vT\vV_N\right)^3\right),
 \\[2ex]
\Var\left(A_{3}\right)\leq\frac{\left(\prod\limits_{i=1}^a {n_i\choose 6}-\prod\limits_{i=1}^a\binom {n_i-6} {6}\right)}{ \prod\limits_{i=1}^a\binom {n_i} {6}}\cdot 27\tr^3\left(\left( \vT\vV_N\right)^2\right)
.\end{array}\]

was proven, as well as
 \[\begin{array}{l}
 \E\left({A_3^\star}(\Upsilon)\right)=\tr\left(\left(\vT\vV_N\right)^3\right),
 \\[2ex]
\Var\left({A_3^\star}(\Upsilon)\right)\leq\left(1-\left(1-\frac{1}{\Upsilon}\right)\cdot\prod\limits_{i=1}^a\frac{ \binom{n_i-6} {6}}{\binom{n_i} {6}}\right)\cdot 27\tr^3\left(\left(\vT\vV_N\right)^2\right)
.\end{array}\] 
\end{proof}

The relations between sample size and the number of groups necessary for using $A_1, A_2$ and $A_3$ and their subsampling version all have the same root. The definition of the random vectors $\vZ$ allows different indices for each group. This leads to a huge number of possible combinations and finally to the products required to converge to one. In case of a balanced setting, it would be easy to use the same indices for all groups and solve this difficulty. But in case of unbalanced sample sizes it would be only possible to chose indices from $\N_{n_{\min}}$, with $n_{\min}:=\min(n_1,...,n_a)$. Thereby many observations from larger groups would never be taken into account. To avoid this, the observations of each group are permutated multiple times. For each permutation, the quadratic forms were summed up, and after enough permutations, the estimator is divided by the number of summations. Another random vector must be introduced first to make this approach formal correct. The random vector $\pi_{j,i}$ represents a random permutation of the numbers $1,\dots,n_i,$ where for different $i$ or $j$ permutations $\pi_{j,i}$ are independent  and $\pi_{j,i}(l) $ denotes its $l$-th component of the vector. This leads to

\[\vZ^{\vpi_j}_{(\ell_{1},\ell_2)} :=\vZ_{\left((\pi_{j,1}({\ell_{1}}),\pi_{j,2}({\ell_1}),\dots,\pi_{j,a}({\ell_1}))^\top,(\pi_{j,1}({\ell_{2}}),\pi_{j,2}({\ell_2}),\dots,\pi_{j,a}({\ell_1}))^\top\right)}.\]

Based on this, depending on the number of permutations $\Upsilon_1\in \N$ estimators for $\tr(\vT\vV_N)$, $\tr\left((\vT\vV_N)^2\right)$ and $\tr\left((\vT\vV_N)^3\right)$  are given through

 \[B_{1}\left(\Upsilon\right)=\frac 1 {2\Upsilon}\sum\limits_{j=1}^{\Upsilon}\sum\limits_{\ell_1<\ell_2=1}^{n_{\min}}
\frac{{\vZ^{\vpi_j}_{(\ell_1,\ell_2)}}^\top \vT{\vZ^{\vpi_j}_{(\ell_1,\ell_2)}} }{\binom{n_{\min}}{2}}, \]

 \[B_{2}\left(\Upsilon\right)=\frac 1 {4\Upsilon_1}\sum\limits_{j=1}^{\Upsilon}\sum\limits_{\begin{footnotesize}\substack{\ell_{1}\neq \dots \neq  \ell_{4}=1\\\ell_{1}<\ell_{2},\dots,\ell_{3}< \ell_{4}}\end{footnotesize}}^{n_{\min}}
\frac{\left[{\vZ^{\vpi_j}_{(\ell_1,\ell_2)}}^\top \vT{\vZ^{\vpi_j}_{(\ell_3,\ell_4)}}\right]^2 }{6 \cdot \binom{n_{\min}}{4}}\]
and 
 \[B_{3}\left(\Upsilon\right)=\frac 1 {8\Upsilon}\sum\limits_{j=1}^{\Upsilon}\sum\limits_{\ell_1\neq...\neq \ell_6=1}^{n_{\min}}
\frac{{\vZ^{\vpi_j}_{(\ell_1,\ell_2)}}^\top \vT{\vZ^{\vpi_j}_{(\ell_3,\ell_4)}}\Cdot {\vZ^{\vpi_j}_{(\ell_3,\ell_4)}}^\top \vT{\vZ^{\vpi_j}_{(\ell_5,\ell_6)}}\cdot {\vZ^{\vpi_j}_{(\ell_5,\ell_6)}}^\top \vT{\vZ^{\vpi_j}_{(\ell_1,\ell_2)}} }{\frac{n_{\min}!}{(n_{\min}-6)!}}.\]

The choice of $\Upsilon$ should include how unbalanced the setting is. For a balanced setting, no permutation is necessary and therefore, we can choose $\Upsilon=1$. But if one group is clearly smaller than the others, a higher value leads to more consistent usage of the observations.

\begin{Le}\label{estimators4}
For the above defined estimators $B_1(\Upsilon)$, $B_2(\Upsilon)$ and $B_3(\Upsilon)$, independent of the chosen value $\Upsilon$, it holds under the frameworks  \eqref{asframe1}-\eqref{asframe3}
\begin{itemize}
\item[a)]
$B_1(\Upsilon)$ is an unbiased, dimensionally stable and ratio-consistent estimator for the unknown $\tr(\vT\vV_N)$.
\item[b)]
$B_2(\Upsilon)$ is an unbiased, dimensionally stable and ratio-consistent estimator for the unknown $\tr\left((\vT\vV_N)^2\right)$.
\item[c)]
$B_3(\Upsilon)$ is an unbiased estimator for the unknown $\tr\left((\vT\vV_N)^3\right)$, which additionally fulfils $\Var(B_3(\Upsilon))\leq 
\Lan(n_{\min}^{-1})\cdot \Lan\left(\tr^2\left(\vT\vV_N\right)^3\right)$.
\end{itemize}
\end{Le}

\begin{proof}

Since the distribution of the vector $\vZ$ not depends on  the concrete indices, the  unbiasedness of the estimators follows by
\[\begin{array}{ll}\E(B_{1}\left(\Upsilon\right))&=\frac 1 {2\Upsilon_1}\sum\limits_{j=1}^{\Upsilon}\sum\limits_{\ell_1<\ell_2=1}^{n_{\min}}
\frac{\E\left({\vZ^{\vpi_j}_{(\ell_1,\ell_2)}}^\top \vT{\vZ^{\vpi_j}_{(\ell_1,\ell_2)}}\right) }{\binom{n_{\min}}{2}}
\\&=\frac 1 {2\Upsilon}\sum\limits_{j=1}^{\Upsilon}\sum\limits_{\ell_1<\ell_2=1}^{n_{\min}}
\frac{\E\left({\vZ_{(\veins,\boldsymbol{2})}}^\top \vT\vZ_{(\veins,\boldsymbol{2})}\right) }{\binom{n_{\min}}{2}}\\
&=\tr(\vT\vV_N),
\end{array}\]

 \[\begin{array}{ll}\E(B_{2}\left(\Upsilon\right))&=\binom{n_{\min}}{4}^{-1}\frac 1 {6\cdot 4\Upsilon}\sum\limits_{j=1}^{\Upsilon}\sum\limits_{\begin{footnotesize}\substack{\ell_{1}\neq \dots \neq  \ell_{4}=1\\\ell_{1}<\ell_{2},\dots,\ell_{3}< \ell_{4}}\end{footnotesize}}^{n_{\min}}
\E\left(\left[{\vZ^{\vpi_j}_{(\ell_1,\ell_2)}}^\top \vT{\vZ^{\vpi_j}_{(\ell_3,\ell_4)}}\right]^2\right) 
\\
&=\binom{n_{\min}}{4}^{-1}\frac 1 {6\cdot 4\Upsilon}\sum\limits_{j=1}^{\Upsilon}\sum\limits_{\begin{footnotesize}\substack{\ell_{1}\neq \dots \neq  \ell_{4}=1\\\ell_{1}<\ell_{2},\dots,\ell_{3}< \ell_{4}}\end{footnotesize}}^{n_{\min}}
\E\left(\left[{\vZ_{(\veins,\boldsymbol{2})}}^\top \vT{\vZ_{(\boldsymbol{3},\boldsymbol{4})}}\right]^2\right)
\\
&=\tr\left((\vT\vV_N)^2\right),
\end{array}\]
and 
 \[\begin{array}{ll}\E(B_{3}\left(\Upsilon\right))&=\frac 1 {8\Upsilon}\sum\limits_{j=1}^{\Upsilon}\sum\limits_{\ell_1\neq...\neq \ell_6=1}^{n_{\min}}
\frac{\E\left({\vZ^{\vpi_j}_{(\ell_1,\ell_2)}}^\top \vT{\vZ^{\vpi_j}_{(\ell_3,\ell_4)}}\Cdot {\vZ^{\vpi_j}_{(\ell_3,\ell_4)}}^\top \vT{\vZ^{\vpi_j}_{(\ell_5,\ell_6)}}\cdot {\vZ^{\vpi_j}_{(\ell_5,\ell_6)}}^\top \vT{\vZ^{\vpi_j}_{(\ell_1,\ell_2)}}\right) }{\frac{n_{\min}!}{(n_{\min}-6)!}}
\\&=\frac 1 {8\Upsilon}\sum\limits_{j=1}^{\Upsilon}\sum\limits_{\ell_1\neq...\neq \ell_6=1}^{n_{\min}}
\frac{\E\left({\vZ_{(\boldsymbol{1},\boldsymbol{2})}}^\top \vT{\vZ_{(\boldsymbol{3},\boldsymbol{4})}}\Cdot {\vZ_{(\boldsymbol{3},\boldsymbol{4})}}^\top \vT{\vZ_{(\boldsymbol{5},\boldsymbol{6})}}\cdot {\vZ_{(\boldsymbol{5},\boldsymbol{6})}}^\top \vT{\vZ_{(\boldsymbol{1},
\boldsymbol{2})}}\right) }{\frac{n_{\min}!}{(n_{\min}-6)!}}\\
&=\tr\left((\vT\vV_N)^3\right).\end{array}\]

For dimensional stability, we now calculate the variance of these estimators, using similar approaches and inequalities as for the previous estimators. Based on the fact that all groups use the same indices, the number of remaining index combinations simplifies, and we calculate

\[\begin{array}{ll}\Var(B_{1}\left(\Upsilon\right))&\leq\frac 1 {4\Upsilon^2}\cdot \binom{n_{\min}}{2}^{-2}\sum\limits_{j_1,j_2=1}^{\Upsilon}\Var\left(\sum\limits_{\ell_1\neq\ell_2=1}^{n_{\min}}{\vZ^{\vpi_1}_{(\ell_1,\ell_2)}}^\top \vT{\vZ^{\vpi_1}_{(\ell_1,\ell_2)}} \right)

\\[2.1ex]&=\frac 1 4 \cdot \left(\frac{ \binom{n_{\min}}{2} - \binom{n_{\min}-2}{2} }{ \binom{n_{\min}}{2}}\right)\Var\left({\vZ^{\vpi_1}_{(1,2)}}^\top \vT{\vZ^{\vpi_1}_{(1,2)}} \right)
\\[2.1ex]&=\frac 1 4 \cdot  \left(\frac{ \binom{n_{\min}}{2} - \binom{n_{\min}-2}{2} }{ \binom{n_{\min}}{2}}\right)\Var\left({\vZ_{(\boldsymbol{1},\boldsymbol{2})}}^\top \vT{\vZ_{(\boldsymbol{1},\boldsymbol{2})}} \right)

\\[2.1ex]
&=\Lan(n_{\min}^{-1})\cdot \Lan\left(\tr^2\left(\vT\vV_N\right)\right).
\end{array}\]\\

Similar calculations for the other estimators lead to
\[\begin{array}{ll}\Var(B_{2}\left(\Upsilon\right))&\leq\frac 1 {16\Upsilon^2}\Cdot \left(6 \binom{n_{\min}}{4}\right)^{-2}\sum\limits_{j_1,j_2=1}^{\Upsilon}\Var\left(\sum\limits_{\ell_1\neq...\neq \ell_4=1}^{n_{\min}}
\left[{\vZ^{\vpi_1}_{(\ell_1,\ell_2)}}^\top \vT{\vZ^{\vpi_1}_{(\ell_3,\ell_4)}}\right]^2\right) 
\\[1.5ex]&=\frac 1 {16}\Cdot  \left(\frac{\binom{n_{\min}}{4} -\binom{n_{\min}-4}{4}}{\binom{n_{\min}}{4} }\right)\Var\left(
\left[{\vZ^{\vpi_1}_{(1,2)}}^\top \vT{\vZ^{\vpi_1}_{(3,4)}}\right]^2\right)

\\[2.5ex]
&=\Lan(n_{\min}^{-1})\cdot \Lan\left(\tr^2\left(\left(\vT\vV_N\right)^2\right)\right),
\end{array}\]

and 

 \[\begin{array}{ll}\Var(B_{3}\left(\Upsilon\right))&\leq\frac 1 {64\Upsilon^2}\sum\limits_{j_1,j_2=1}^{\Upsilon}\Var\left(\sum\limits_{\ell_1\neq...\neq \ell_6=1}^{n_{\min}}
\frac{{\vZ^{\vpi_1}_{(\ell_1,\ell_2)}}^\top \vT{\vZ^{\vpi_1}_{(\ell_3,\ell_4)}}\Cdot {\vZ^{\vpi_1}_{(\ell_3,\ell_4)}}^\top \vT{\vZ^{\vpi_1}_{(\ell_5,\ell_6)}}\cdot {\vZ^{\vpi_1}_{(\ell_5,\ell_6)}}^\top \vT{\vZ^{\vpi_1}_{(\ell_1,\ell_2)}} }{\frac{n_{\min}!}{(n_{\min}-6)!}}\right)
\\[3.2ex]&=
\frac 1 {64}\Cdot  \left(\frac{\frac {  n_{\min}! }{\left(n_{\min}-6\right)!} -\frac {  \left(n_{\min}-6\right)! }{\left(n_{\min}-12\right)!} }{\frac {  n_{\min}! }{\left(n_{\min}-6\right)!} }\right)
\Var\left(
\frac{{\vZ^{\vpi_1}_{(1,2)}}^\top \vT{\vZ^{1}_{(3,4)}}\Cdot {\vZ^{\vpi_1}_{(3,4)}}^\top \vT{\vZ^{\vpi_1}_{(5,6)}}\cdot {\vZ^{\vpi_1}_{(5,6)}}^\top \vT{\vZ^{\vpi_1}_{(1,2)}} }{\frac{n_{\min}!}{(n_{\min}-6)!}}\right)

\\[2ex]
&=\Lan(n_{\min}^{-1})\cdot \Lan\left(\tr^2\left(\vT\vV_N)^3\right)\right).\end{array}\]
\end{proof}

To reduce the required computation time, we also define subsampling versions of these estimators given by
 \[B_{1}^\star\left(\Upsilon_1,\Upsilon_2\right)=\frac 1 {2\Upsilon_1\Upsilon_2}\sum\limits_{j=1}^{\Upsilon_1}\sum\limits_{\upsilon=1}^{\Upsilon_2}
{\vZ^{\vpi_j}_{(\vsigma_1(\upsilon,2),\vsigma_2(\upsilon,2))}}^\top \vT{\vZ^{\vpi_j}_{(\vsigma_1(\upsilon,2),\vsigma_2(\upsilon,2))}} , \]

 \[B_{2}^\star\left(\Upsilon_1,\Upsilon_2\right)=\frac 1 {4\Upsilon_1\Upsilon_2}\sum\limits_{j=1}^{\Upsilon_1}\sum\limits_{\upsilon=1}^{\Upsilon_2}
{\left[{\vZ^{\vpi_j}_{(\vsigma_1(\upsilon,4),\vsigma_2(\upsilon,4))}}^\top \vT{\vZ^{\vpi_j}_{(\vsigma_3(\upsilon,4),\vsigma_4(\upsilon,4))}}\right]^2 }\]
and 
 \[B_{3}^\star\left(\Upsilon_1,\Upsilon_2\right)=\frac 1 {8\Upsilon_1\Upsilon_1}\sum\limits_{j=1}^{\Upsilon_1}\sum\limits_{\upsilon=1}^{\Upsilon_2}{\Lambda^{\vpi_j}(\upsilon;1,2,3,4)\Cdot \Lambda^{\vpi_j}(\upsilon;3,4,5,6)\cdot \Lambda^{\vpi_j}(\upsilon;5,6,1,2) }\]
with
\[\Lambda^{\vpi_j}(\upsilon;\ell_1,\ell_2,\ell_3,\ell_4):={\vZ^{\vpi_j}_{(\vsigma_{\ell_1}(\upsilon,6),\vsigma_{\ell_2}(\upsilon,6))}}^\top \vT{\vZ^{\vpi_j}_{(\vsigma_{\ell_3}(\upsilon,6),\vsigma_{\ell_4}(\upsilon,6))}}.\]

\begin{Le}\label{estimators5}
In each of the asymptotic frameworks \eqref{asframe1}-\eqref{asframe3} the subsampling estimators $B_{1}^\star\left(\Upsilon_1,\Upsilon_2\right)$, $B_{2}^\star\left(\Upsilon_1,\Upsilon_2\right)$ and $B_{3}^\star\left(\Upsilon_1,\Upsilon_2\right)$ have the same properties like $B_1(\Upsilon),B_2(\Upsilon)$ and $B_3(\Upsilon)$ .
\end{Le}
\begin{proof}
Again the concrete indices have no influence on the expectation values, and therefore the unbiasedness follows by
 \[\begin{array}{ll}\E\left(B_{1}^\star\left(\Upsilon_1,\Upsilon_2\right)\right)&=\frac 1 {2\Upsilon_1\Upsilon_2}\sum\limits_{j=1}^{\Upsilon_1}\sum\limits_{\upsilon=1}^{\Upsilon_2}
{\E\left({\vZ^{\vpi_j}_{(\vsigma_1(\upsilon,2),\vsigma_2(\upsilon,2))}}^\top \vT{\vZ^{\vpi_j}_{(\vsigma_1(\upsilon,2),\vsigma_2(\upsilon,2))}} \right)}
\\&=\frac 1 {2\Upsilon_1\Upsilon_2}\sum\limits_{j=1}^{\Upsilon_1}\sum\limits_{\upsilon=1}^{\Upsilon_2}
{\E\left({\vZ_{(\veins,\boldsymbol{2})}}^\top \vT\vZ_{(\veins,\boldsymbol{2})}\right)}
\\[1.3ex]&=\tr(\vT\vV_N),
\end{array} \]

 \[\begin{array}{ll}\E\left(B_{2}^\star\left(\Upsilon_1,\Upsilon_2\right)\right)&=\frac 1 {4\Upsilon_1\Upsilon_2}\sum\limits_{j=1}^{\Upsilon_1}\sum\limits_{\upsilon=1}^{\Upsilon_2}
\E\left({\left[{\vZ^{\vpi_j}_{(\vsigma_1(\upsilon,4),\vsigma_2(\upsilon,4))}}^\top \vT{\vZ^{\vpi_j}_{(\vsigma_3(\upsilon,4),\vsigma_4(\upsilon,4))}}\right]^2 }\right)
\\&=\frac 1 {4\Upsilon_1\Upsilon_2}\sum\limits_{j=1}^{\Upsilon_1}\sum\limits_{\upsilon=1}^{\Upsilon_2}
\E\left(\left[{\vZ_{(\veins,\boldsymbol{2})}}^\top \vT{\vZ_{(\boldsymbol{3},\boldsymbol{4})}}\right]^2\right)
\\&=\tr\left((\vT\vV_N)^2\right)
\end{array} \]
and 
 \[\begin{array}{ll}\E\left(B_{3}^\star\left(\Upsilon_1,\Upsilon_2\right)\right)&=\frac 1 {8\Upsilon_1\Upsilon_2}\sum\limits_{j=1}^{\Upsilon_1}\sum\limits_{\upsilon=1}^{\Upsilon_2}{\E\left(\Lambda^{\vpi_j}(\upsilon;1,2,3,4)\Cdot \Lambda^{\vpi_j}(\upsilon;3,4,5,6)\cdot \Lambda^{\vpi_j}(\upsilon;5,6,1,2) \right)}\\
 &=\frac 1 {8\Upsilon_1\Upsilon_2}\sum\limits_{j=1}^{\Upsilon_1}\sum\limits_{\upsilon=1}^{\Upsilon_2}{\E\left({\vZ_{(\boldsymbol{1},\boldsymbol{2})}}^\top \vT{\vZ_{(\boldsymbol{3},\boldsymbol{4})}}\Cdot {\vZ_{(\boldsymbol{3},\boldsymbol{4})}}^\top \vT{\vZ_{(\boldsymbol{5},\boldsymbol{6})}}\cdot {\vZ_{(\boldsymbol{5},\boldsymbol{6})}}^\top \vT{\vZ_{(\boldsymbol{1},
\boldsymbol{2})}}\right) }
\\&=\tr\left((\vT\vV_N)^3\right) 
  .\end{array} \]

For the variance, we consider\\

 $\begin{array}{ll}
&\Var\left(\E\left(\frac 1 {2\Upsilon_2}\sum\limits_{\upsilon=1}^{\Upsilon_2}
{\vZ^{\vpi_j}_{(\vsigma_1(\upsilon,2),\vsigma_2(\upsilon,2))}}^\top \vT{\vZ^{\vpi_j}_{(\vsigma_1(\upsilon,2),\vsigma_2(\upsilon,2))}} \Big\lvert\F(\vsigma(\Upsilon_2,2))\right)\right)\\[1ex]
=&\Var\left(\tr\left(\vT\vV_N\right)\right)=0.

\end{array}$\\\\ and therefore with the inequalities from \cite{sattler2018} \\\\
$\begin{array}{ll}
&\Var\left(\frac 1 {2\Upsilon_2}\sum\limits_{\upsilon=1}^{\Upsilon_2}
{\vZ^{\vpi_j}_{(\vsigma_1(\upsilon,2),\vsigma_2(\upsilon,2))}}^\top \vT{\vZ^{\vpi_j}_{(\vsigma_1(\upsilon,2),\vsigma_2(\upsilon,2))}} \right)
\\[2ex]
=&0+\frac 1 {4\Upsilon_2^2}\E\left(\Var\left(\sum\limits_{\upsilon=1}^{\Upsilon_2}
{\vZ^{\vpi_j}_{(\vsigma_1(\upsilon,2),\vsigma_2(\upsilon,2))}}^\top \vT{\vZ^{\vpi_j}_{(\vsigma_1(\upsilon,2),\vsigma_2(\upsilon,2))}}\Big\lvert\F(\vsigma(\Upsilon_2,2))\right)\right)\end{array}$\\

$\begin{array}{ll}
{\leq}&\frac 1 {4\Upsilon_2^2}\E\left(\sum\limits_{(j,\ell)\in \N_{\Upsilon_2} \times \N_{\Upsilon_2}  \setminus M(\Upsilon_2,\vsigma(\upsilon,2))} \Var\left( {\vZ_{(\vsigma_1(j,2),\vsigma_2(j,2))}}^\top \vT{\vZ_{(\vsigma_1(j,2),\vsigma_2(j,2))}}|\F(\vsigma(\Upsilon_2,2))\right)\right)
\\[3.2ex]
=&\frac{\E\left(|\N_{\Upsilon_2}\times \N_{\Upsilon_2} \setminus M(\Upsilon_2,\vsigma(\upsilon,2))|\right)}{\Upsilon_2^2}\cdot \frac{\Var\left( {\vZ_{(\vsigma_1(j,2),\vsigma_2(j,2))}}^\top \vT{\vZ_{(\vsigma_1(j,2),\vsigma_2(j,2))}}\right)}{4\Upsilon_2^2}
\\[1.0ex]
{\leq}&
\left(1-\left(1-\frac{1}{\Upsilon_2}\right)\cdot \frac{ \binom{n_{\min}-2} {2}}{\binom{n_{\min}} {2}}\right)\cdot \tr^2\left(\vT\vV_N\right).
\end{array}$\\\\

and therefore for the whole estimator\\\\
$\begin{array}{ll}
\Var\left(B_{1}^\star\left(\Upsilon_1,\Upsilon_2\right)\right)&

\leq \frac{1}{\Upsilon_1^2}\left(\sum\limits_{j=1}^{\Upsilon_1}\sqrt{\Var\left(\frac 1 {2\Upsilon_2}\sum\limits_{\upsilon=1}^{\Upsilon_2}
{\vZ^{\vpi_j}_{(\vsigma_1(\upsilon,2),\vsigma_2(\upsilon,2))}}^\top \vT{\vZ^{\vpi_j}_{(\vsigma_1(\upsilon,2),\vsigma_2(\upsilon,2))}}\right)}\right)^2
\\[2.5ex]
&\leq \frac{1}{\Upsilon_1^2}\left(\sum\limits_{j=1}^{\Upsilon_1}\sqrt{\left(1-\left(1-\frac{1}{\Upsilon_2}\right)\cdot \frac{ \binom{n_{\min}-2} {2}}{\binom{n_{\min}} {2}}\right)\cdot \tr^2\left(\vT\vV_N\right)}\right)^2
\\[2.5ex]

&=\left(1-\left(1-\frac{1}{\Upsilon_2}\right)\cdot \frac{ \binom{n_{\min}-2} {2}}{\binom{n_{\min}} {2}}\right)\cdot \tr^2\left(\vT\vV_N\right).
\end{array}$\\\\

The same steps for the other estimators, together with results from the proof of \Cref{estimators4} lead to

$\begin{array}{ll}
\Var\left(B_{2}^\star\left(\Upsilon_1,\Upsilon_2\right)\right)&
=\left(1-\left(1-\frac{1}{\Upsilon_2}\right)\cdot \frac{ \binom{n_{\min}-4} {4}}{\binom{n_{\min}} {4}}\right)\cdot  \Lan\left(\tr^2\left(\left(\vT\vV_N\right)^2\right)\right).
\end{array}$\\\\
and 

$\begin{array}{ll}
\Var\left(B_{3}^\star\left(\Upsilon_1,\Upsilon_2\right)\right)&
=\left(1-\left(1-\frac{1}{\Upsilon_2}\right)\cdot \frac{ \binom{n_{\min}-6} {6}}{\binom{n_{\min}} {6}}\right)\cdot  \Lan\left(\tr^3\left(\left( \vT\vV_N\right)^2\right)\right).
\end{array}$

\end{proof}
This Lemma is focused on an increasing number of $\vUpsilon_2$, while the number of permutations $\vUpsilon_1$ has no influence.

To use the observations from bigger groups more evenly, choosing $\Upsilon_2=1$ to have a maximal number of mixtures would be reasonable. From a computational sight, this choice would lead to more effort and, therefore, higher calculation time. So the choice of $\Upsilon_1$ and $\Upsilon_2 $ is always a balance between the property of the estimators and time. With additional notation and extension of results from \cite{sattler2018}, the above Lemma can also be formulated under the condition $\Upsilon_1\Upsilon_2\to \infty$.

\begin{proof}[Proof of Theorem \ref{estimators3}]
With the results from Lemma \ref{estimators4}  the estimators $B_1,B_2$ and $B_3$ have the same properties as $A_1$, $A_2$ and $A_3$, without their additional conditions on the relation between samples sizes and the number of groups. Therefore they possess all properties, used for the proof of  \Cref{Theorem4} and \Cref{estimators2}. In this way the results follows directly  replace the estimators and identically for  $B_1^\star(\Upsilon_1,\Upsilon_2)$, $B_2^\star(\Upsilon_1,\Upsilon_2)$ and $B_3^\star(\Upsilon_1,\Upsilon_2)$ based on \ref{estimators5}.
\end{proof}

\bibliographystyle{apalike}
\bibliographystyle{unsrtnat}
\newpage

\phantomsection
\addcontentsline{toc}{chapter}{Bibliography}
\bibliography{Literatur}

\end{document}